\documentclass[a4paper,reqno]{amsart}
\usepackage{amssymb}
\usepackage{dsfont}
\usepackage{hyperref}

\def\N{\mathbb{N}}
\def\Z{\mathbb{Z}}
\def\R{\mathbb{R}}

\def\hau{\mathcal{H}}
\def\leb{\mathcal{L}}
\def\pack{\mathcal{P}}
\def\prepack{P}

\def\Hdim{\operatorname{dim_{H}}}
\def\Pdim{\operatorname{dim_{P}}}
\def\UBdim{\operatorname{\overline{dim}_{B}}}
\def\Hcodim{\operatorname{codim_{H}}}

\def\as{\mbox{a.s.}}
\newcommand{\diam}[1]{|#1|}

\def\ind{\mathds{1}}
\def\dd{\mathrm{d}}
\def\ee{\mathrm{e}}
\def\ph{\varphi}
\def\eps{\varepsilon}

\def\pprec{\ensuremath{\prec\!\!\!\prec}}

\def\Bseq{\mathrm{B}}
\def\Xseq{\mathrm{X}}
\def\aseq{\mathrm{a}}
\def\kseq{\mathrm{k}}
\def\rseq{\mathrm{r}}
\def\xseq{\mathrm{x}}

\def\Acal{\mathcal{A}}
\def\Bcal{\mathcal{B}}
\def\Ecal{\mathcal{E}}
\def\Fcal{\mathcal{F}}
\def\Gcal{\mathcal{G}}
\def\Hcal{\mathcal{H}}
\def\Kcal{\mathcal{K}}
\def\Mcal{\mathcal{M}}
\def\Ncal{\mathcal{N}}
\def\Pcal{\mathcal{P}}
\def\Tcal{\mathcal{T}}
\def\Ucal{\mathcal{U}}
\def\Vcal{\mathcal{V}}

\newcommand{\gen}[1]{\ensuremath{\langle #1\rangle}}

\def\T{\mathbb{T}}

\def\prob{\mathbb{P}}
\def\esp{\mathbb{E}}

\def\dist{d}
\newcommand{\opball}[2]{\mathrm{B}(#1,#2)}
\newcommand{\clball}[2]{\overline{\mathrm{B}}(#1,#2)}

\newcommand{\pare}[1]{\ensuremath{\overleftarrow{#1}}}
\newcommand{\adh}[1]{\ensuremath{\overline{#1}}}

\newtheorem{thm}{Theorem}

\newtheorem{conj}{Conjecture}
\newtheorem{prp}{Proposition}
\newtheorem{lem}{Lemma}

\theoremstyle{remark}

\newtheorem{pb}{Problem}
\newtheorem*{mpb}{Mahler's Problem}

\theoremstyle{definition}
\newtheorem{df}{Definition}

\begin{document}

\title[Diophantine approximation on the Cantor set]{Metric Diophantine approximation\\ on the middle-third Cantor set}

\author{Yann Bugeaud}
\address{Yann Bugeaud\\ Universit\'e de Strasbourg\\ Math\'ematiques\\ 7 rue Ren\'e Descartes\\ 67084 Strasbourg, France}
\email{bugeaud@math.unistra.fr}

\author{Arnaud Durand}
\address{Arnaud Durand\\ Universit\'e Paris-Sud\\ Laboratoire de Math\'ematiques d'Orsay -- UMR 8628\\ B\^atiment 425\\ 91405 Orsay Cedex, France}
\email{arnaud.durand@math.u-psud.fr}


\date{\today}

\begin{abstract}
Let $\mu\geq 2$ be a real number and let $\Mcal(\mu)$ denote the set of real numbers approximable at order at least $\mu$ by rational numbers. More than eighty years ago, Jarn\'\i k and, independently, Besicovitch established that the Hausdorff dimension of $\Mcal(\mu)$ is equal to $2/\mu$. We investigate the size of the intersection of $\Mcal(\mu)$ with Ahlfors regular compact subsets of the interval $[0, 1]$. In particular, we propose a conjecture for the exact value of the dimension of $\Mcal(\mu)$ intersected with the middle-third Cantor set and give several results supporting this conjecture. We especially show that the conjecture holds for a natural probabilistic model that is intended to mimic the distribution of the rationals. The core of our study relies heavily on dimension estimates concerning the set of points lying in an Ahlfors regular set and approximated at a given rate by a system of random points.
\end{abstract}

\maketitle

%%%%%%%%%%%%%%%%%%%%%%%%%%
%%%%%%%%%%%%%%%%%%%%%%%%%%
\section{Introduction}\label{sec:intro}
%%%%%%%%%%%%%%%%%%%%%%%%%%
%%%%%%%%%%%%%%%%%%%%%%%%%%

In Section~2 of his paper {\em Some suggestions for further research}, Mahler~\cite{Mahler:1984zr} posed the following question.

\begin{mpb}
How close can irrational elements of Cantor's set be approximated by rational numbers
\begin{enumerate}
\item[(i).] in Cantor's set, and
\item[(ii).] by rational numbers not in Cantor's set?
\end{enumerate}
\end{mpb}

Here, {\em Cantor's set} is the middle-third Cantor set, that is, the set of all real numbers of the form $a_1 3^{-1} + a_2 3^{-2} + \ldots + a_i 3^{-i} + \ldots$, with $a_{i}\in\{0,2\}$ for every integer $i\geq 1$. This set is denoted by $K$ in all what follows. In other words, Mahler asked whether there are elements in the Cantor set with any prescribed {\it irrationality exponent}; see also~\cite[Problem~35]{Bugeaud:2004fk}. Recall that the irrationality exponent $\mu(\xi)$ of an irrational real number $\xi$ is defined by
\[
\mu(\xi)=\sup\left\{\mu\in\R \:\Biggl|\: \left| \xi - \frac{p}{q} \right| < \frac{1}{q^{\mu}} \text{~for i.m.~} (p,q)\in\Z\times\N \right\},
\]
where ``i.m.'' stands for infinitely many. The irrationality exponent of every irrational number is greater than or equal to two and it is precisely equal to two for Lebesgue almost all real numbers, see~\cite[Section~1]{Schmidt:1971gf}. 
Furthermore, when $\xi$ is a rational number, we set $\mu(\xi)=1$. 

As a first step towards Mahler's question, Weiss~\cite{Weiss:2001mz} established that the irrationality exponent is also equal to two for almost every point in the Cantor set $K$, with respect to the standard measure thereon. Levesley, Salp and Velani~\cite{Levesley:2007pd} constructed explicit elements of $K$ having a prescribed irrationality exponent: for $\mu$ greater than or equal to $(3+\sqrt{5})/2$, they showed that
\begin{equation}\label{eq:museries}
\mu \left( \sum_{j=1}^{\infty} 2 \cdot 3^{-\lfloor \mu^{j} \rfloor} \right) = \mu,
\end{equation}
where $\lfloor\,\cdot\,\rfloor$ denotes the integer part function. Subsequently, Bugeaud~\cite{Bugeaud:2008uq} used the theory of continued fractions to prove that~(\ref{eq:museries}) also holds for all $\mu\geq 2$ and that there are uncountably many elements in $K$ with any prescribed irrationality exponent greater than or equal to two. This gives a satisfactory answer to Mahler's question; however, unfortunately, the method does not yield any information on the size of the set of points in the Cantor set whose irrationality exponent is at least equal to a given real number $\mu>2$. The starting point of the present paper is to investigate this problem, by considering the Hausdorff dimension of the intersection set $\Mcal(\mu)\cap K$, where
\[
\Mcal(\mu)=\{\xi\in\R  \:|\: \mu(\xi)\geq\mu\}\,;
\]
we refer to Section~\ref{subsec:notations} below for the necessary recalls on the notion of Hausdorff dimension. The size of the two sets forming the above intersection is very well known. First, the Hausdorff dimension of the middle-third Cantor set $K$ satisfies
\begin{equation}\label{eq:dimCantor}
\Hdim K=\kappa \qquad\text{with}\qquad \kappa=\frac{\log 2}{\log 3},
\end{equation}
see for instance~\cite{Falconer:2003oj}. Second, a famous result established independently by Jarn\'\i k~\cite{Jarnik:1929mf} and Besicovitch~\cite{Besicovitch:1934ly} asserts that
\begin{equation}\label{eq:JarnikBesicovitch}
\forall \mu\geq 2 \qquad \Hdim \Mcal(\mu)=\frac{2}{\mu}.
\end{equation}
Furthermore, for every $\mu\geq 2$, the set of all real numbers with irrationality exponent exactly equal to $\mu$ has the same Hausdorff dimension as the set $\Mcal(\mu)$.

Besides the irrationality exponent, we also consider the exponents $v_{b}$ which first appeared in~\cite{Amou:2010fk}, but were already implicitly used in~\cite{Levesley:2007pd}. They provide information on the lengths of blocks of digits $0$ (or of digits $b-1$) occurring in the expansion of an irrational real number $\xi$ to the integer base $b\geq 2$, and are defined by
\[
v_{b}(\xi)=\sup\left\{v\in\R \:\bigl|\: \|b^{j}\xi\|< b^{-vj} \text{~for i.m.~} j\in\N \right\},
\]
where $\|\,\cdot\,\|$ denotes the distance to the nearest integer. When $\xi$ is rational, it is convenient to adopt the convention that $v_{b}(\xi)=0$. It is easy to see that, for any irrational real number $\xi$, 
\begin{equation}\label{eq:mugeqvb}
\mu(\xi) \geq v_b(\xi)+1.
\end{equation}
However, these inequalities are rarely sharp. As a matter of fact, the exponent $v_{b}$ vanishes for Lebesgue almost every real number. Furthermore, for any real number $v\geq 0$, it is well-known that the Hausdorff dimension of the set
\begin{equation}\label{eq:defVbv}
\Vcal_{b}(v)=\{\xi\in\R \:|\: v_{b}(\xi)\geq v \}
\end{equation}
is equal to $1/(v+1)$.

The triadic analog of the above question is then to determine the Hausdorff dimension of the intersection of sets $\Vcal_{3}(v)\cap K$. This has been performed by Levesley, Salp and Velani~\cite{Levesley:2007pd}, thereby shedding some new light on Mahler's problem. To be specific, Corollary~1 in~\cite{Levesley:2007pd} asserts that
\begin{equation}\label{eq:dimKv3geq}
\Hdim(\Vcal_{3}(v)\cap K)=\frac{\kappa}{v+1},
\end{equation}
which can be seen as the product of the dimension of $K$ by that of $\Vcal_{3}(v)$. It is also proved in~\cite{Levesley:2007pd} that~(\ref{eq:dimKv3geq}) still holds when $\Vcal_{3}(v)$ is replaced by the set of all real numbers $\xi$ such that $v_{3}(\xi)=v$. This shows that there exist points in the Cantor set that can be approximated at any prescribed order by rational numbers whose denominators are powers of three, and gives a very satisfactory answer to the triadic analog of Mahler's question. Note however that things are much easier with the exponent $v_{3}$ than with the irrationality exponent $\mu$, mainly because of the following reason: if a point $\xi\in K$ is approximated at a rate $v>1$ by a triadic number $p/3^{j}$, that is, if
\[
\left|\xi-\frac{p}{3^{j}}\right|<3^{-vj},
\]
then this triadic number $p/3^{j}$ necessarily lies in $K$. Let us also mention that Fishman and Simmons~\cite{Fishman:2012uq} recently extended~(\ref{eq:dimKv3geq}) to the following more general situation: the sets $\Vcal_{3}(v)$ are replaced by $\Vcal_{b}(v)$ for an arbitrary prime value of $b$, and the middle-third Cantor set $K$ is replaced by the fractal set composed of all the real numbers whose $b$-ary digits belong to some fixed subset of $\{0,\ldots,b-1\}$.

With the help of~(\ref{eq:mugeqvb}), let us now remark that the set $\Mcal(\mu)$ contains the set $\Vcal_{3}(\mu-1)$ for any value of $\mu\geq 2$. Along with~(\ref{eq:dimKv3geq}), this readily implies that
\begin{equation}\label{eq:dimKMmulow}
\Hdim(\Mcal(\mu)\cap K)\geq\Hdim(\Vcal_{3}(\mu-1)\cap K)=\frac{\kappa}{\mu}.
\end{equation}
Thus, in view of~(\ref{eq:JarnikBesicovitch}), the Hausdorff dimension of the intersection of the sets $K$ and $\Mcal(\mu)$ is bounded from below by half the product of their dimensions. As regards the upper bound, Pollington and Velani~\cite{Pollington:2005fk} used a covering argument due to Weiss~\cite{Weiss:2001mz} to establish that for every $\mu\geq 2$,
\begin{equation}\label{eq:PoVe}
\Hdim(\Mcal(\mu)\cap K)\leq\frac{2\kappa}{\mu},
\end{equation}
so that the dimension of the intersection is bounded from above by the product of the dimensions. We also refer to the work of Kristensen~\cite{Kristensen:2006ys} for a similar result.

In view of the aforementioned results, Levesley, Salp and Velani~\cite{Levesley:2007pd} speculate at the end of their paper that the dimension of the intersection of the sets $K$ and $\Mcal(\mu)$ is equal to the product of their dimensions, namely,
\begin{equation}\label{eq:conjLSV}
\Hdim(\Mcal(\mu)\cap K)=\frac{2\kappa}{\mu}.
\end{equation}
They also believe that the following weaker statement holds:
\begin{equation}\label{eq:conjLSVweak}
\lim_{\mu\to 2^{+}} \Hdim(\Mcal(\mu)\cap K) = \kappa.
\end{equation}
The starting point of the present work is to discuss the validity of~(\ref{eq:conjLSV}) and~(\ref{eq:conjLSVweak}). We agree with~(\ref{eq:conjLSVweak}) but disagree with~(\ref{eq:conjLSV}). As a matter of fact, we now propose another conjectural dimension for the intersection set $\Mcal(\mu)\cap K$.

\begin{conj}\label{conj:main}
For any real number $\mu\geq 2$, the set of points in the middle-third Cantor set whose irrationality exponent is at least $\mu$ satisfies
\begin{equation}\label{eq:conjmain}
\Hdim(\Mcal(\mu)\cap K)=\max \left\{ \frac{2}{\mu}+\kappa-1, \frac{\kappa}{\mu} \right\}.
\end{equation}
\end{conj}

We believe that a proof of Conjecture~\ref{conj:main} is very difficult, and requires a deep understanding of the distribution of rational points near the middle-third Cantor set. Note that, according to this conjecture, the Hausdorff dimension of the intersection set $\Mcal(\mu)\cap K$ exhibits a ``phase transition'' at the critical value
\begin{equation}\label{eq:defmucrit}
\mu=\frac{2-\kappa}{1-\kappa}=\frac{\log(9/2)}{\log(3/2)}.
\end{equation}
The approach that we develop hereunder actually suggests the following behaviors:
\begin{itemize}
\item Below this critical value, the rational numbers that belong to the Cantor set, or are very close thereto, do not play a privileged role in the approximation of the points of the Cantor set; as in many generic situations where there is no particular interplay between two sets, the codimension of their intersection is thus the sum of their codimensions, namely,
\[
\Hcodim(\Mcal(\mu)\cap K)=\Hcodim\Mcal(\mu)+\Hcodim K=2-\frac{2}{\mu}-\kappa,
\]
see {\em e.g.}~\cite[Chapter~8]{Falconer:2003oj} for such generic situations.
\item Above the critical value, the aforementioned rational numbers become predominant when approximating the points of the Cantor set; the dimension is thus equal to the lower bound~(\ref{eq:dimKMmulow}) obtained by Levesley, Salp and Velani~\cite{Levesley:2007pd} which corresponds to restricting the approximating rationals to being the triadic endpoints of the intervals occurring in the construction of the Cantor set.
\end{itemize}

Various arguments supporting this conjecture are given in Section~\ref{sec:argconj} below. Therein, we begin by giving heuristic arguments, and we also present a doubly metric point of view that shows further evidence for Conjecture~\ref{conj:main}. We also present a randomized version of the above problem; this consists in replacing the approximating rational numbers by random points that are intended to mimic the distribution of the rationals while taking into account the fact that some rationals fall into the Cantor set exactly, or are very close to it. In particular, we show that Conjecture~\ref{conj:main} is verified for this random model, see Section~\ref{subsec:probconjmain}.

The motivation behind the study of such randomized models starts from the following observation. From the viewpoint of the metric theory of Diophantine approximation, the points with rational coordinates and a sequence of random points chosen independently and uniformly in a given nonempty compact set satisfy a lot of common properties: for example, they both give rise to homogeneous ubiquitous systems and the sets of points that they approximate share the same size and large intersection properties, see~\cite{Durand:2007uq,Durand:2010fk} and the references therein. Pushing further the analogy, we also put forward a conjecture for the exponents $v_{b}$ when $b$ is not a power of three.

\begin{conj}\label{conj:vbnot3}
Let us assume that $b$ is not a power of three. Then:
\begin{enumerate}
\item for any real $\xi\in K$,
\[
0\leq v_{b}(\xi)\leq\frac{\kappa}{1-\kappa}\,;
\]
\item for any real number $v\in[0,\kappa/(1-\kappa)]$,
\[
\Hdim(\Vcal_{b}(v)\cap K)=\frac{1}{v+1}+\kappa-1.
\]
\end{enumerate} 
\end{conj}

In that situation, we expect that there is so little interplay between the expansions to the bases $b$ and $3$ that we are in the generic situation where the codimension of the intersection of the sets $K$ and $\Vcal_{b}(v)$ is equal to the sum of their codimensions. This is of course in stark contrast with the formula~(\ref{eq:dimKv3geq}) obtained by Levesley, Salp and Velani in the case where $b=3$. The probabilistic arguments supporting Conjecture~\ref{conj:vbnot3} are given in Section~\ref{subsec:probvb}.

In Section~\ref{sec:randomgene}, which is actually the core of this paper, we put the study of the probabilistic counterparts of the aforementioned number theoretical open questions in a more general context. Specifically, we place ourselves for convenience on the circle $\T=\R/\Z$, endowed with the usual quotient distance $\dist$, and we develop a general theory for the size of the intersection sets $\Ecal(\Xseq,\rseq)\cap G$, where
\[
\Ecal(\Xseq,\rseq) = \left\{\xi\in\T \:\bigl|\: d(\xi,X_{n}) < r_{n} \text{~for i.m.~} n\geq 1 \right\}
\]
and $G$ is a nonempty compact subset of $\T$. Here, $\Xseq=(X_{n})_{n \geq 1}$ is a sequence of random variables in $\T$ and $\rseq=(r_{n})_{n\geq 1}$ is a sequence of real numbers in $(0,1]$. In particular, we give the probability with which the random set $\Ecal(\Xseq,\rseq)$ intersects the compact set $G$, and we analyze the value of the Hausdorff measures of the intersection $\Ecal(\Xseq,\rseq)\cap G$ for general gauge functions. The random points $X_{n}$ are chosen according to the Lebesgue measure on the circle, and are often supposed to have very few dependence on each other.

In Section~\ref{sec:approxindep}, we consider the particular case where the approximating points $X_{n}$ are stochastically independent. This is what ultimately enables us to establish the probabilistic counterpart of Conjecture~\ref{conj:main} mentioned above. In Section~\ref{sec:approxfrac}, we allow some weak dependence between those points, namely, we assume that $X_{n}$ is the fractional part of $a_{n}X$, where $X$ is chosen according to the Lebesgue measure on $[0,1)$ and $(a_{n})_{n\geq 1}$ is a sequence of positive integers that grows sufficiently rapidly. Our findings lead to the following metrical statement: if the sequence $(a_{n})_{n\geq 1}$ grows fast enough, for instance if $a_{n+1}\geq n^{\log\log n} a_{n}$ for $n\geq 2$, then for Lebesgue almost every real $\alpha$ and for every real $\nu\geq 1$,
\[
\Hdim\left\{\xi\in K\:\biggl|\: \|a_{n}\alpha-\xi\|<\frac{1}{n^{\nu}} \text{~for i.m.~} n\geq 1\right\}=\frac{1}{\nu}+\kappa-1
\]
if this value is nonnegative; otherwise, the above set is empty. We refer to Theorem~\ref{thm:approxfrac} for details. If the real numbers $\xi$ are not restricted to belong to the Cantor set, then one recovers a much easier situation already studied by various authors including Bugeaud~\cite{Bugeaud:2003ye}, Schmeling and Troubetzkoy~\cite{Schmeling:2003}, Fan, Schmeling, and Troubetzkoy~\cite{Fan:2007fj}, and also by Liao and Seuret~\cite{Liao:2013fk}. Most of the proofs are postponed to Sections~\ref{sec:proofmain} and~\ref{sec:fracperc}, while Section~\ref{sec:remprob} is devoted to concluding observations and a brief discussion on further problems.

%%%%%%%%%%%%%%%%%%%%%%%%%%
%%%%%%%%%%%%%%%%%%%%%%%%%%
\section{Various arguments supporting the conjectures}\label{sec:argconj}
%%%%%%%%%%%%%%%%%%%%%%%%%%
%%%%%%%%%%%%%%%%%%%%%%%%%%

We begin by giving heuristic arguments aiming at supporting Conjecture~\ref{conj:main}. We then introduce a doubly metric point of view that shows further evidence for this statement to hold. We also put forward a randomized version of the problem, and we show that Conjecture~\ref{conj:main} holds for this probabilistic model. Finally, we present an analogous random model meant for supporting Conjecture~\ref{conj:vbnot3}. Before proceeding, we first need to set up some notations and give some recalls that will be useful throughout the paper.

%%%%%%%%%%%%%%%%%%%%%%%%%%
\subsection{Notations and recalls}\label{subsec:notations}
%%%%%%%%%%%%%%%%%%%%%%%%%%

For convenience, we shall almost always place ourselves on the circle $\T=\R/\Z$. As a matter of fact, the function $\xi\mapsto\mu(\xi)$ that maps a real number to its irrationality exponent is one-periodic, so we may consider it only on the interval $[0,1)$. With a slight abuse, we shall identify the elements of this interval with those of the circle. As a result, the irrationality exponent $\mu(\xi)$ of an irrational point $\xi\in\T$ may be written in the form
\[
\mu(\xi)=\sup\left\{\mu\in\R \:\bigl|\: \dist(\xi,p/q) < q^{-\mu} \text{~for i.m.~} (p,q)\in\Pcal \right\},
\] 
where $\dist$ is the usual quotient distance on the circle. Likewise, if $\xi$ is a rational number in the circle, then $\mu(\xi)=1$. In the above formula, $\Pcal$ denotes the collection of all pairs $(p,q)$ of integers such that $0\leq p<q$ and $\gcd(p,q)=1$, so that every rational number in the circle $\T$ may be written in the form $p/q$ for a unique pair $(p,q)\in\Pcal$. Then, the intersection of the set $\Mcal(\mu)$ with the interval $[0,1)$  may be identified with the set of all points $\xi\in\T$ such that $\mu(\xi)\geq\mu$. For simplicity, this subset of the circle will still be denoted by $\Mcal(\mu)$; this is also the image of the original set under the projection modulo one. Plainly, the same kind of analogy holds for the exponents $v_{b}(\xi)$ and the sets $\Vcal_{b}(v)$. In addition, we still denote by $K$ the image of the middle-third Cantor set under the projection modulo one.

Let us now give a brief account of the notion of Hausdorff and packing measures and dimensions; we refer for instance to~\cite{Falconer:2003oj,Mattila:1995fk} for further details. Let $d$ be a positive integer. Let $g$ denote a {\em gauge function}, that is, a nondecreasing right-continuous function defined on $[0,\infty)$ which vanishes at zero, and only at zero. The Hausdorff $g$-measure of a subset $E$ of $\T^{d}$ is defined by
\[
\hau^{g}(E)=\lim_{\delta\downarrow 0}\uparrow \hau^{g}_{\delta}(E)
\qquad\text{with}\qquad
\hau^{g}_{\delta}(E)=\inf\sum_{n=1}^{\infty} g(\diam{U_{n}}),
\]
where $\diam{\,\cdot\,}$ stands for diameter and the infimum is taken over all sequences $(U_{n})_{n\geq 1}$ of subsets of $\T^{d}$ verifying $E\subseteq\bigcup_{n} U_{n}$ and $\diam{U_{n}}<\delta$ for all $n\geq 1$. We shall sometimes assume that the gauge function is {\em doubling}, that is, satisfies $g(2r)\leq C g(r)$ for all $r>0$ and some $C>0$.

We shall also make use of the packing $g$-measures. Recall that the packing premeasure associated with a gauge function $g$ is defined by
\[
\prepack^{g}(E)=\lim_{\delta\downarrow 0}\downarrow \prepack^{g}_{\delta}(E)
\qquad\text{with}\qquad
\prepack^{g}_{\delta}(E)=\sup\sum_{n=1}^{\infty} g(\diam{B_{n}}),
\]
where the supremum is taken over all sequences $(B_{n})_{n\geq 1}$ of disjoint closed balls of $\T^{d}$ centered in $E$ and with diameter less than $\delta$. The packing $g$-measure of a set $E$ is then defined by
\[
\pack^{g}(E)=\inf_{E\subseteq\bigcup_{n} U_{n}}\sum_{n=1}^{\infty} \prepack^{g}(U_{n}).
\]
It is know that $\pack^{g}$, as well as $\hau^{g}$, is a Borel measure on the torus $\T^{d}$. However, the premeasures $\prepack^{g}$ are only finitely subadditive.

When the gauge function $g$ is of the form $r\mapsto r^{s}$ with $s>0$, it is customary to let $\hau^{s}$, $\prepack^{s}$ and $\pack^{s}$ stand for $\hau^{g}$, $\prepack^{g}$ and $\pack^{g}$, respectively. These gauge functions give rise to the notion of Hausdorff and packing dimensions. To be specific, the Hausdorff dimension of a nonempty set $E\subseteq\T^d$ is defined by
\[
\Hdim E=\sup\{ s\in (0,d) \:|\: \hau^{s}(E)=\infty \}=\inf\{ s\in (0,d) \:|\: \hau^{s}(E)=0 \},
\]
with the convention that $\sup\emptyset=0$ and $\inf\emptyset=d$. Likewise, the packing dimension $\Pdim E$ is defined by replacing the Hausdorff measure $\hau^{s}$ by the packing measure $\pack^{s}$ in the above formula. Moreover, one recovers the upper box-counting dimension $\UBdim E$ by considering the premeasures $\prepack^{s}$ instead of $\hau^{s}$. All these dimensions thus enable one to give an abridged description of the size properties of $E$. When the set $E$ is empty, we adopt the convention that these dimensions are all equal to $-\infty$.

Finally, so as to make some of our statements more tangible, we often work under the following regularity assumption when considering compact subsets of the circle.

\begin{df}[Ahlfors regularity]\label{df:Ahlfors}
A compact subset $G$ of the circle is Ahlfors regular with dimension $\gamma\in(0,1]$ if there exists a real $c>0$ such that
\[
\forall x\in G \quad \forall r>0 \qquad \frac{r^{\gamma}}{c} \leq \hau^{\gamma}(G\cap\opball{x}{r})\leq c r^{\gamma},
\]
where $\opball{x}{r}$ is the open arc centered at $x$ with length $2r$.
\end{df}

In view of the mass distribution principle for Hausdorff and packing measures, if a compact set $G$ is Ahlfors regular with dimension $\gamma$, we then have
\[
0<\hau^{\gamma}(G)\leq\pack^{\gamma}(G)\leq\prepack^{\gamma}(G)<\infty,
\]
so that the Hausdorff, box-counting and packing dimensions of $G$ coincide and are all equal to $\gamma$, see~\cite{Falconer:2003oj,Mattila:1995fk}. We refer to~\cite{David:1997uq} for more details on Ahlfors regularity and important examples of regular sets; in particular, it is clear that $\T$ is regular with dimension one and it is well known that the set $K$ is regular with dimension $\kappa$ given by~(\ref{eq:dimCantor}).

%%%%%%%%%%%%%%%%%%%%%%%%%%
\subsection{Heuristic arguments supporting Conjecture~\ref{conj:main}}\label{subsec:heuristic}
%%%%%%%%%%%%%%%%%%%%%%%%%%

Prior to stating rigorous results, let us begin by giving some loose arguments towards Conjecture~\ref{conj:main}. Note that for large values of $\mu$, the conjectured dimension coincides with the lower bound~(\ref{eq:dimKMmulow}) resulting from the work of Levesley, Salp and Velani~\cite{Levesley:2007pd}. Therefore, the chief novelty brought by Conjecture~\ref{conj:main} concerns the small values of the approximation rate; the main purpose of our discussion is then to explain why we expect~(\ref{eq:conjmain}) to hold, especially for small values of $\mu$. We actually focus our heuristic arguments towards the upper bound on the Hausdorff dimension, because it is certainly the easiest to get a feel on.

To begin with, note that the density of a given subset $\Pcal'$ of $\Pcal$ may be measured by means of the parameter
\begin{equation}\label{eq:defsigmaPK}
\sigma(\Pcal')=\limsup_{j\to\infty}\frac{1}{j}\log_{3}\#(\Pcal'\cap\Pcal^{j}).
\end{equation}
Here, $\log_{3}$ is the base three logarithm and $\Pcal^{j}$ is the set of pairs $(p,q)\in\Pcal$ such that $3^{j}\leq q<3^{j+1}$. Using standard estimates on the growth of Euler's totient function~\cite[Theorem~13.14]{Apostol:1976uq}, one easily checks that $\log_{3}\#\Pcal^{j}$ is equivalent to $2j$ as $j$ goes to infinity. Thus, $\sigma(\Pcal')$ is bounded from above by two, and the closer $\sigma(\Pcal')$ is to this bound, the denser $\Pcal'$ is in $\Pcal$.

In view of making the connection with Conjecture~\ref{conj:main}, let us consider the set $\Pcal_{K}^{0}$ formed by the rational numbers that belong to the Cantor set $K$, namely,
\begin{equation}\label{eq:defPK0}
\Pcal_{K}^{0}=\{(p,q)\in\Pcal\:|\:p/q\in K\}.
\end{equation}
For any integer $j\geq 1$, it is easy to see that the set $\Pcal_{K}^{0}\cap\Pcal^{j}$ contains the pairs $(2+3a_{1}+\ldots+3^{j-1}a_{j-1},3^{j})$, for all possible choices of $a_{1},\ldots,a_{j-1}\in\{0,2\}$. As a consequence, $\sigma(\Pcal_{K}^{0})\geq\kappa$, where $\kappa$ is the Hausdorff dimension of $K$, see~(\ref{eq:dimCantor}). It is conjectured that there are not considerably much more elements in $\Pcal_{K}^{0}\cap\Pcal^{j}$ than those specified above. To be specific, basing themselves on computer simulations, Broderick, Fishman and Reich made the following conjecture: for all $\eps>0$,
\begin{equation}\label{eq:conjBFR1}
\#(\Pcal_{K}^{0}\cap\Pcal^{j})={\rm O}(2^{(1+\eps)j}) \qquad\text{as}\qquad j\to\infty,
\end{equation}
see~\cite[Conjecture~1]{Broderick:2011fk}; we also refer to~\cite{Fishman:2012uq} for heuristic arguments supporting this conjecture. The validity of~(\ref{eq:conjBFR1}) would straightforwardly imply that
\begin{equation}\label{eq:conjBFR1bis}
\sigma(\Pcal_{K}^{0})=\kappa.
\end{equation}

Moreover, let us consider a point $\xi$ in the Cantor set $K$ and a pair $(p,q)$ in the set $\Pcal^{j}$, and assume that $\dist(\xi,p/q)<q^{-\mu}$. Then, it is clear that the pair $(p,q)$ belongs to the set
\[
\Pcal_{K}^{\mu,j}=\{(p,q)\in\Pcal^{j}\:|\:\dist(p/q,K)<3^{-\mu j}\},
\]
where $\dist(\,\cdot\,,K)$ denotes the distance to the Cantor set $K$. The points at a distance less than $3^{-\mu j}$ from $K$ form a set with Lebesgue measure of the order of $(3^{-\mu j})^{\kappa-1}$; this is due to the fact that $K$ is Ahlfors regular with dimension $\kappa$. Thus, assuming that the rational numbers $p/q$, for $(p,q)\in\Pcal^{j}$, are evenly spread in the circle, this value of the Lebesgue measure should give the proportion of pairs in $\Pcal^{j}$ that belong to $\Pcal_{K}^{\mu,j}$. This would imply that $\#\Pcal_{K}^{\mu,j}$ is of the order of $3^{(2-(1-\kappa)\mu)j}$, up to logarithmic factors. However, this estimate is too stringent when $\mu$ is large; we must indeed take into account the fact that $\Pcal_{K}^{\mu,j}$ necessarily contains $\Pcal_{K}^{0}\cap\Pcal^{j}$, which gives a lower bound on its cardinality. Combined with~(\ref{eq:conjBFR1bis}), the previous arguments result in the following conjecture:
\begin{equation}\label{eq:conjcardPKmuj}
\limsup_{j\to\infty}\frac{1}{j}\log_{3}\#\Pcal_{K}^{\mu,j}\leq\max\{2-(1-\kappa)\mu,\kappa\}.
\end{equation}
Verifying this conjecture would of course require a very good understanding of the distribution of the rational numbers lying near the Cantor set. The conjecture suggests that when $\mu$ is larger than the critical value defined by~(\ref{eq:defmucrit}), the condition defining $\Pcal_{K}^{\mu,j}$ becomes so strict that the $3^{-\mu j}$-neighborhood of $K$ cannot contain considerably more rational numbers than $K$ itself. This is probably what lies at the root of the ``phase transition'' phenomenon mentioned in Section~\ref{sec:intro}.

Finally, for any real $\eps>0$ and any integer $j_{0}\geq 1$, we plainly have
\[
\Mcal(\mu)\cap K\subseteq\bigcup_{j=j_{0}}^{\infty}\bigcup_{(p,q)\in\Pcal_{K}^{\mu-\eps,j}}\opball{p/q}{3^{-(\mu-\eps)j}}.
\]
We may then apply the Hausdorff-Cantelli lemma, and deduce that the Hausdorff dimension of $\Mcal(\mu)\cap K$ is bounded from above by 
any positive real number $s$ for which the series
\[
\sum_{j=1}^{\infty}\#\Pcal_{K}^{\mu-\eps,j}(3^{-(\mu-\eps)j})^{s}
\]
converges. Assuming that the conjectured estimate~(\ref{eq:conjcardPKmuj}) holds, and letting $\eps$ go to zero, we end up with the formula given in Conjecture~\ref{conj:main}.

%%%%%%%%%%%%%%%%%%%%%%%%%%
\subsection{A doubly metric point of view: rotating the Cantor set}\label{subsec:doubmet}
%%%%%%%%%%%%%%%%%%%%%%%%%%

In view of~(\ref{eq:dimKMmulow}), the intersection of the sets $\Mcal(\mu)$ and $K$ cannot be too small, because the rational endpoints of the middle-third Cantor set contribute to its Hausdorff dimension in a very special way. We believe however that these rational endpoints lose their privileged status in the approximation when the exponent $\mu$ is small. A drastic way of artificially removing this privileged status is to rotate the Cantor set by a generic angle $\alpha$; indeed, the endpoints of the set $\alpha+K$ are generically not rational anymore, and thus may not be used in the approximation. Here, $\alpha+K$ denotes the image under the circle rotation with angle $\alpha$ of the set $K$. In order to study the size of the intersection of the set $\Mcal(\mu)$ with the rotated Cantor set $\alpha+K$, we may adopt a doubly metric point of view: we analyze the set of pairs $(\xi,\alpha)$ in the torus such that $\xi$ belongs to $\Mcal(\mu)\cap (\alpha+K)$. We also develop the same approach for the exponent $v_{b}$ related to the expansion to a given base $b\geq 2$. In both cases, the formulae that we obtain for the dimension are similar to those expressed in Conjectures~\ref{conj:main} and~\ref{conj:vbnot3}.

%%%%%%%%%%%%%%%%%%%%%%%%%%
\subsubsection{The irrationality exponent}
%%%%%%%%%%%%%%%%%%%%%%%%%%

A straightforward adaptation of the arguments of Weiss~\cite{Weiss:2001mz} and Kristensen~\cite{Kristensen:2006ys} implies that the bound~(\ref{eq:PoVe}) holds uniformly after rotating the Cantor set by an arbitrary angle; specifically, for every real $\mu\geq 2$ and every angle $\alpha\in\T$,
\[
\Hdim(\Mcal(\mu)\cap(\alpha+K))\leq\frac{2\kappa}{\mu}.
\]
However, the next result gives a generic upper bound that is much more stringent than the above uniform one.

\begin{thm}\label{thm:genupbndMcalmualphaK}
The following holds for Lebesgue almost every angle $\alpha\in\T$:
\begin{enumerate}
\item for any irrational point $\xi\in\alpha+K$,
\[
2\leq\mu(\xi)\leq\frac{2}{1-\kappa}\,;
\]
\item for any real number $\mu\in[2,2/(1-\kappa)]$,
\[
\Hdim(\Mcal(\mu)\cap(\alpha+K))\leq\frac{2}{\mu}+\kappa-1.
\]
\end{enumerate}
\end{thm}

The above corollary shows that codimension of the intersection of the sets $\Mcal(\mu)$ and $\alpha+K$ is generically at least the sum of their codimensions. As mentioned previously, such a situation is expected to occur when there is no particular interplay between the two sets whose intersection is being taken, see {\em e.g.}~\cite[Chapter~8]{Falconer:2003oj}. In other words, the points of the generically rotated Cantor set do not have a specific status with respect to the approximation by rationals. For small values of the exponent $\mu$, we expect that this observation remains valid when the Cantor set is not even rotated. As a matter of fact, the bound given by the above corollary then matches that of Conjecture~\ref{conj:main}.

The situation is very different for large values of $\mu$. Indeed, the above result ensures that the Hausdorff dimension of $\Mcal(\mu)\cap (\alpha+K)$ is at most zero for generic values of $\alpha$, but~(\ref{eq:dimKMmulow}) shows that the dimension is positive when $\alpha$ vanishes. Therefore, when $\mu$ is large, the rationals that belong to the Cantor set, or are very close thereto, become predominant when approximating the points of the Cantor set. Still, they lose their privileged status when the Cantor set is rotated in a generic manner.

Let us now establish Theorem~\ref{thm:genupbndMcalmualphaK}. This result actually follows from the general statements that we give in Section~\ref{sec:randomgene} below. To be more specific, let $(p_{n},q_{n})_{n\geq 1}$ denote an enumeration of the set $\Pcal$ for which $(q_{n})_{n \geq 1}$ is nondecreasing. Furthermore, let $\rseq=(r_{n})_{n\geq 1}$ be the sequence defined by $r_{n}=1/q_{n}$, and let $\Xseq=(X_{n})_{n\geq 1}$ be the circle valued sequence defined by $X_{n}=p_{n}/q_{n}-\alpha$. Note that when $\alpha$ is chosen according to the Lebesgue measure, the points $X_{n}$ are uniformly distributed in the circle. It is now clear that we are in the framework considered in Section~\ref{sec:randomgene}. Indeed, observe that for any real $\eps>0$,
\begin{equation}\label{eq:inclMmualphaK}
\Mcal(\mu)\cap(\alpha+K)\subseteq\alpha+(\Ecal(\Xseq,\rseq^{\mu-\eps})\cap K),
\end{equation}
where the set $\Ecal(\Xseq,\rseq^{\mu-\eps})$ is defined as in~(\ref{eq:defEcalXr}) and $\rseq^{\mu-\eps}$ denotes the sequence formed by the real numbers $r_{n}^{\mu-\eps}$. Given that $K$ is Ahlfors regular with dimension $\kappa$, we have $\prepack^{\kappa}(G)<\infty$, and we may thus apply Theorem~\ref{thm:majsizegene}. This way, we deduce that with probability one:
\[
\mu>\eps+\frac{2}{1-\kappa} \qquad\Longrightarrow\qquad \Ecal(\Xseq,\rseq^{\mu-\eps})\cap K=\emptyset
\]
and for every real $s>0$,
\[
s>\frac{2}{\mu-\eps}+\kappa-1 \qquad\Longrightarrow\qquad \hau^{s}(\Ecal(\Xseq,\rseq^{\mu-\eps})\cap K)=0.
\]
Along with~(\ref{eq:inclMmualphaK}), these two implications straightforwardly lead to Theorem~\ref{thm:genupbndMcalmualphaK}.

For the sake of completeness, let us also give a proof of the above theorem that does not call upon the general results stated in Section~\ref{sec:randomgene}. We actually obtain a slightly weaker statement than Theorem~\ref{thm:genupbndMcalmualphaK}, namely, for every real $\mu\geq 2$ and Lebesgue almost every angle $\alpha\in\T$,
\begin{equation}\label{eq:genupbndMcalmualphaKweak}
\Hdim(\Mcal(\mu)\cap(\alpha+K))\leq\max\left\{\frac{2}{\mu}+\kappa-1,0\right\}.
\end{equation}
However, the advantage of this alternative proof is that it exhibits a connection with a doubly metric statement which has its own interest. To proceed, observe that for any angle $\alpha$, the set $\Mcal(\mu)\cap(\alpha+K)$ may be regarded as the intersection of the set
\[
\Mcal^{\times}(\mu)=\left\{ (\xi,\alpha)\in\T^{2} \:\bigl|\: \xi\in\alpha+K \text{ and } \mu(\xi)\geq\mu \right\}.
\]
with the line of $\T^{2}$ formed by the points whose second coordinate is equal to $\alpha$. Thus, applying a classical slicing result such as~\cite[Corollary~7.10]{Falconer:2003oj}, we deduce that for Lebesgue almost every angle $\alpha\in\T$,
\[
\Hdim(\Mcal(\mu)\cap(\alpha+K))\leq\max\{\Hdim\Mcal^{\times}(\mu)-1,0\}.
\]
In order to obtain~(\ref{eq:genupbndMcalmualphaKweak}), we are thus reduced to estimating the Hausdorff dimension of $\Mcal^{\times}(\mu)$, which is the purpose of the next statement.

\begin{prp}
For any real number $\mu\geq 2$,
\[
\Hdim\Mcal^{\times}(\mu)=\frac{2}{\mu}+\kappa.
\]
\end{prp}

\begin{proof}
To begin with, let us observe that the mapping $\Psi:(\xi,\alpha)\mapsto (\xi-\alpha,\alpha)$ from $\T^{2}$ onto itself is bi-Lipschitz and satisfies
\[
\Psi(\Mcal^{\times}(\mu))=\left\{ (\xi,\alpha)\in\T^{2} \:\bigl|\: \xi\in K \text{ and } \alpha\in -\xi+\Mcal(\mu) \right\}.
\]
Hence, $\Mcal^{\times}(\mu)$ has the same Hausdorff dimension as the above set, which is easier to handle. Moreover, the set $K$ is Ahlfors regular with dimension $\kappa$, so we may adapt the proof of~\cite[Proposition~7.9]{Falconer:2003oj} to show that for all $s>\kappa$, there exists a constant $c>0$, depending only  on $K$ and $s$, such that for any Borel subset $E$ of $\T^{2}$,
\[
\hau^{s}(E)\geq c\int_{K}\hau^{s-\kappa}(E\cap L_{\xi_{0}})\,\hau^{\kappa}(\dd\xi_{0}),
\]
where $L_{\xi_{0}}$ is the set of points $(\xi,\alpha)\in\T^{2}$ such that $\xi=\xi_{0}$. We now apply this result to the set $\Psi(\Mcal^{\times}(\mu))$. It is easy to see that for each $\xi_{0}\in\T$, there is a simple isometry which maps the intersection with $L_{\xi_{0}}$ of this last set onto the product set $\{0\}\times\Mcal(\mu)$. Consequently,
\[
\hau^{s}(\Psi(\Mcal^{\times}(\mu)))\geq c\hau^{s-\kappa}(\Mcal(\mu))\hau^{\kappa}(K).
\]
If $s-\kappa$ is less than $2/\mu$, we deduce from~(\ref{eq:JarnikBesicovitch}) that $\hau^{s-\kappa}(\Mcal(\mu))$ is infinite, so that the Hausdorff measure appearing in the left-hand side is infinite as well. It follows that the Hausdorff dimension of $\Mcal^{\times}(\mu)$ is bounded from below by $2/\mu+\kappa$.

For the reverse inequality, it suffices to find an appropriate covering of $\Psi(\Mcal^{\times}(\mu))$. To proceed, let us consider a positive real number $\eps$, and a point $(\xi,\alpha)$ in this last set. The irrationality exponent of $\alpha+\xi$ is then larger than $\mu-\eps$, so that for any integer $q_{0}\geq 1$, there is a rational number $p/q$ with denominator at least $q_{0}$ such that $\dist(\alpha+\xi,p/q)<q^{-\mu+\eps}$. Moreover, let $j(q)$ be the integer part of the base three logarithm of $q^{\mu-\eps}/2$. The set $K$ is naturally covered by $2^{j(q)}$ closed arcs with length $3^{-j(q)}$; let $x_{j(q),0},\ldots,x_{j(q),2^{j(q)}-1}$ denote their centers. Given that $\xi$ belongs to $K$, we have $\dist(\xi,x_{j(q),k})\leq 3^{-j(q)}/2$ for some $k$. Furthermore, making use of the triangle inequality, we deduce that
\[
\dist\Bigl(\alpha,\frac{p}{q}-x_{j(q),k}\Bigr)\leq\dist\Bigl(\frac{p}{q}-\alpha,\xi\Bigr)+\dist(\xi,x_{j(q),k})<q^{-\mu+\eps}+\frac{3^{-j(q)}}{2}\leq 3^{-j(q)}.
\]
If $\T^{2}$ is equipped with the product distance, it follows that the point $(\xi,\alpha)$ belongs to the open ball with radius $3^{-j(q)}$ centered at $(x_{j(q),k},p/q-x_{j(q),k})$, which is denoted by $B_{q,p,k}$. As a result, for any $\eps>0$ and $q_{0}\geq 1$,
\[
\Psi(\Mcal^{\times}(\mu))\subseteq\bigcup_{q=q_{0}}^{\infty}\bigcup_{p=0}^{q-1}\bigcup_{k=0}^{2^{j(q)}-1} B_{q,p,k}.
\]
Let $s$ and $\delta$ denote two positive real numbers. For $q_{0}$ large enough, we deduce from the above covering that
\[
\hau^{s}_{\delta}(\Psi(\Mcal^{\times}(\mu)))\leq\sum_{q=q_{0}}^{\infty} q2^{j(q)} (2\cdot 3^{-j(q)})^{s}\leq 3^{s}2^{2s-\kappa}\sum_{q=q_{0}}^{\infty} q^{1+(\mu-\eps)(\kappa-s)},
\]
and the last series converges when $s$ is larger than $2/(\mu-\eps)+\kappa$. The required upper bound on the Hausdorff dimension of $\Mcal^{\times}(\mu)$ now follows straightforwardly.
\end{proof}

%%%%%%%%%%%%%%%%%%%%%%%%%%
\subsubsection{The exponents $v_{b}$}
%%%%%%%%%%%%%%%%%%%%%%%%%%

Given an integer $b\geq 2$ and a real $v\geq 0$, the above method enables one to study the size properties of the set
\[
\Vcal^{\times}_{b}(v)=\left\{ (\xi,\alpha)\in\T^{2} \:\bigl|\: \xi\in\alpha+K \text{ and } v_{b}(\xi)\geq v \right\},
\]
which is the analog of the set $\Mcal^{\times}(\mu)$ for the exponent $v_{b}$ related to the expansion to the base $b$. To be precise, making the obvious changes to the last proof, one easily checks that the Hausdorff dimension of this set is given by
\[
\Hdim\Vcal^{\times}_{b}(v)=\frac{1}{v+1}+\kappa.
\]
As a consequence, for almost every angle $\alpha\in\T$ in the sense of Lebesgue measure, we also plainly have
\begin{equation}\label{eq:genupbndVbvalphaKweak}
\Hdim(\Vcal_{b}(v)\cap(\alpha+K))\leq\max\left\{\frac{1}{v+1}+\kappa-1, 0\right\}.
\end{equation}
Again, this bound is much more stringent than the uniform bound 
which follows from the arguments of Weiss~\cite{Weiss:2001mz} and Kristensen~\cite{Kristensen:2006ys}, specifically,
\[
\forall\alpha\in\T \qquad \Hdim(\Vcal_{b}(v)\cap(\alpha+K))\leq\frac{\kappa}{v+1}.
\]
Finally, making use of the results of Section~\ref{sec:randomgene}, we may establish the following analog of Theorem~\ref{thm:genupbndMcalmualphaK}, thereby obtaining a slightly more precise statement than~(\ref{eq:genupbndVbvalphaKweak}).

\begin{thm}\label{thm:genupbndVbvalphaK}
Let $b\geq 2$ be an integer. The following holds for Lebesgue almost every angle $\alpha\in\T$:
\begin{enumerate}
\item for any point $\xi\in\alpha+K$,
\[
0\leq v_{b}(\xi)\leq\frac{\kappa}{1-\kappa}\,;
\]
\item for any real number $v\in[0,\kappa/(1-\kappa)]$,
\begin{equation}\label{eq:genupbndVbvalphaK}
\Hdim(\Vcal_{b}(v)\cap(\alpha+K))\leq\frac{1}{v+1}+\kappa-1.
\end{equation}
\end{enumerate}
\end{thm}

When $b$ is equal to three, the above result is of course in stark contrast with the formula~(\ref{eq:dimKv3geq}) obtained by Levesley, Salp and Velani, which corresponds to the case of the original Cantor set where $\alpha=0$. This is again due to the fact that the endpoints of the Cantor set, which are triadic rational numbers, lose their privileged role in the approximation when the Cantor set is rotated.

When $b$ is not a power of three, it is expected that there is very little interaction between the expansions to the bases $b$ and three. Hence, the points of the Cantor set should not satisfy particular properties with respect to the approximation by $b$-adic rationals. Therefore, the Hausdorff dimension of the intersection set $\Vcal_{b}(v)\cap(\alpha+K)$ should be the same for $\alpha=0$ as for a generic value of $\alpha$. Conjecture~\ref{conj:vbnot3} is thus equivalent to the fact that Theorem~\ref{thm:genupbndVbvalphaK} is still valid when $\alpha$ vanishes, and that~(\ref{eq:genupbndVbvalphaK}) is not only an upper bound, but an equality.

%%%%%%%%%%%%%%%%%%%%%%%%%%
\subsection{Conjecture~\ref{conj:main} holds for a probabilistic counterpart of the irrationality exponent}\label{subsec:probconjmain}
%%%%%%%%%%%%%%%%%%%%%%%%%%

As mentioned in Section~\ref{sec:intro}, from the viewpoint of metric number theory, the points with rational coordinates and a sequence of random points chosen independently and uniformly in a given nonempty compact set satisfy a lot of common properties: they both lead to homogeneous ubiquitous systems, and to a variety of sets which share the same size and large intersection properties, see~\cite{Durand:2007uq,Durand:2010fk} and the references therein. Starting from this remark, we shall replace the approximating rational numbers by a sequence of random points which is intended to mimic the distribution of rational numbers and to take into account the fact that some rational numbers fall into the Cantor set exactly, or are very close to it; we shall then show that Conjecture~\ref{conj:main} is verified for this random model.

Let us now detail our model. Let $\Pcal_{K}$ denote a subset of $\Pcal$. In view of making the connection with Conjecture~\ref{conj:main}, we intend the set $\Pcal_{K}$ to contain the pairs $(p,q)$ in $\Pcal$ such that the rational number $p/q$ is exactly in $K$, or very close thereto. In particular, we intend $\Pcal_{K}$ to contain the set $\Pcal_{K}^{0}$ defined by~(\ref{eq:defPK0}). For this reason, and in view of~(\ref{eq:conjBFR1bis}), we assume from now on that $\Pcal_{K}$ is a subset of $\Pcal$ that satisfies
\begin{equation}\label{eq:densityPK}
\sigma(\Pcal_{K})=\kappa.
\end{equation}

We may now consider a family $(Y_{p,q})_{(p,q)\in\Pcal}$ of independent random variables in the circle such that:
\begin{itemize}
\item if $(p,q)\in\Pcal_{K}$, then $Y_{p,q}$ is distributed uniformly in $K$, that is, according to the $\kappa$-dimensional Hausdorff measure restricted to $K$;
\item if $(p,q)\not\in\Pcal_{K}$, then $Y_{p,q}$ is distributed uniformly in $\T$, that is, according to the Lebesgue measure.
\end{itemize}
Instead of considering the approximation by rational numbers, we shall study the approximation by those random points $Y_{p,q}$. Specifically, we are interested in the size properties of the random subsets $\Fcal(\mu)$ of $\T$ defined by
\begin{equation}\label{eq:defFcalmu}
\Fcal(\mu) = \left\{\xi\in\T \:\bigl|\: d(\xi,Y_{p,q}) < q^{-\mu} \text{~for i.m.~} (p,q)\in\Pcal \right\}.
\end{equation}
The mapping $\mu\mapsto\Fcal(\mu)$ is obviously nonincreasing, and for any point $\xi\in\T$, we may define
\begin{equation}\label{eq:defmubullet} 
\mu^{\bullet}(\xi) = \sup\{\mu\geq 0 \:|\: \xi\in\Fcal(\mu)\},
\end{equation}
which is the analog of the irrationality exponent for the approximation by the random points $Y_{p,q}$.

The philosophy behind the above random model is the following. We believe that the pairs $(p,q)\in\Pcal$ such that $p/q$ is exactly in, or very close to, the Cantor set $K$ form a set with density parametrized by $\kappa$. Thus, we choose a subset $\Pcal_{K}$ of $\Pcal$ that is intended to contain those pairs, and we assume that~(\ref{eq:densityPK}) holds. In particular, $\Pcal_{K}$ has low density in $\Pcal$. We then randomize the situation: we replace the vast majority of the rational numbers by random points that are chosen uniformly in the circle $\T$, and we also introduce a slight bias in the distribution in the sense that the rational numbers that are indexed by a pair in $\Pcal_{K}$ are replaced by random points that are chosen uniformly in $K$. The rate of approximation $q^{-\mu}$ is left unchanged.

With the help of the results obtained in Section~\ref{sec:approxindep} below, we may now establish Conjecture~\ref{conj:main} above in this randomized situation. This amounts to proving that~(\ref{eq:conjmain}) holds when the set $\Mcal(\mu)$ is replaced by its random counterpart
\[
\Mcal^{\bullet}(\mu)=\{\xi\in\T \:|\: \mu^{\bullet}(\xi)\geq\mu\}\,;
\]
this results in the following statement. Note that we may apply the results of Section~\ref{sec:approxindep} because the set $K$ is Ahlfors regular with dimension $\kappa$.

\begin{thm}\label{thm:conjmainrand}
The following holds with probability one:
\begin{enumerate}
\item\label{thm:conjmainrand1} for any point $\xi\in K$,
\[
\mu^{\bullet}(\xi)\geq 2\,;
\]
\item\label{thm:conjmainrand2} for any real number $\mu\geq 2$,
\[
\Hdim(\Mcal^{\bullet}(\mu)\cap K)=\max\left\{ \frac{2}{\mu} +  \kappa - 1, \frac{\kappa}{\mu} \right\}.
\]
\end{enumerate}
\end{thm}

The remainder of this section is devoted to the proof of Theorem~\ref{thm:conjmainrand}, modulo Propositions~\ref{prp:approxunifT} and~\ref{prp:approxunifGregular} below. Let $\Fcal_{K}(\mu)$ and $\Fcal_{K^{\complement}}(\mu)$ denote the sets obtained when replacing $\Pcal$ by $\Pcal_{K}$ and its complement $\Pcal_{K}^{\complement}$, respectively, in the definition~(\ref{eq:defFcalmu}) of $\Fcal(\mu)$. Then, we may decompose this last set in the following manner:
\[
\Fcal(\mu)=\Fcal_{K}(\mu)\cup\Fcal_{K^{\complement}}(\mu).
\]
Note that the two sets forming the above union are not necessarily disjoint. This enables us 
to rewrite the exponent $\mu^{\bullet}(\xi)$ in the form
\[
\mu^{\bullet}(\xi)=\max\{\mu^{\bullet}_{K}(\xi),\mu^{\bullet}_{K^{\complement}}(\xi)\},
\]
where the exponents $\mu^{\bullet}_{K}(\xi)$ and $\mu^{\bullet}_{K^{\complement}}(\xi)$ are defined by replacing $\Fcal(\mu)$ by $\Fcal_{K^{\complement}}(\mu)$ and $\Fcal_{K^{\complement}}(\mu)$, respectively, in~(\ref{eq:defmubullet}). The proof of Theorem~\ref{thm:conjmainrand} now reduces to showing the next two lemmas.

\begin{lem}\label{lem:conjmainrandK}
The following holds with probability one:
\begin{enumerate}
\item for any point $\xi\in K$,
\[
\mu^{\bullet}_{K}(\xi)\geq 1\,;
\]
\item for any real number $\mu\geq 1$,
\[
\Hdim\{\xi\in K \:|\: \mu^{\bullet}_{K}(\xi)\geq\mu\}=\frac{\kappa}{\mu}.
\]
\end{enumerate}
\end{lem}

\begin{proof}
Let $\Xseq=(X_{n})_{n\geq 1}$ be a sequence of random points that are independently and uniformly chosen in the Cantor set $K$, and $\rseq=(r_{n})_{n\geq 1}$ be the sequence defined by $r_{n}=1/q_{n}$, where $(p_{n},q_{n})_{n\geq 1}$ is an enumeration of $\Pcal_{K}$ for which $(q_{n})_{n \geq 1}$ is nondecreasing. Note that~(\ref{eq:sumrnnu}) holds with $\rho=\sigma(\Pcal_{K})$, because of the definition~(\ref{eq:defsigmaPK}) of this parameter. It is now easy to see that the sets $\Fcal_{K}(\mu)$ are distributed as the sets $\Ecal(\Xseq,\rseq^{\mu})$ defined as in~(\ref{eq:defEcalXr}). As a consequence, the exponent $\mu^{\bullet}_{K}(\xi)$ is distributed as $\nu_{\Xseq,\rseq}(\xi)$. The result follows from Proposition~\ref{prp:approxunifGregular} below, along with~(\ref{eq:densityPK}).
\end{proof}

\begin{lem}
The following holds with probability one:
\begin{enumerate}
\item for any point $\xi\in K$,
\[
2\leq\mu^{\bullet}_{K^{\complement}}(\xi)\leq\frac{2}{1-\kappa}\,;
\]
\item for any real number $\mu\in[2,2/(1-\kappa)]$,
\[
\Hdim\{\xi\in K \:|\: \mu^{\bullet}_{K^{\complement}}(\xi)\geq\mu\}=\frac{2}{\mu}+ \kappa-1.
\]
\end{enumerate}
\end{lem}

\begin{proof}
The proof is very similar to that of Lemma~\ref{lem:conjmainrandK}. Let $\Xseq=(X_{n})_{n\geq 1}$ denote a sequence of random points that are independently and uniformly chosen in the circle $\T$, and let $\rseq=(r_{n})_{n\geq 1}$ denote the sequence defined by $r_{n}=1/q_{n}$, where $(p_{n},q_{n})_{n\geq 1}$ is an enumeration of $\Pcal_{K}^{\complement}$ for which $q_{n}$ is nondecreasing. One easily checks that~(\ref{eq:sumrnnu}) holds with $\rho=2$ and that the sets $\Fcal_{K^{\complement}}(\mu)$ are distributed as the sets $\Ecal(\Xseq,\rseq^{\mu})$. Hence, the exponent $\mu^{\bullet}_{K^{\complement}}(\xi)$ is distributed as $\nu_{\Xseq,\rseq}(\xi)$, and it just remains to apply Proposition~\ref{prp:approxunifT} below.
\end{proof}

The above approach is quite flexible in the sense that~(\ref{eq:densityPK}) may be adapted in order to fit the true value of $\sigma(\Pcal_{K}^{0})$. In accordance with Broderick, Fishman and Reich~\cite{Broderick:2011fk}, we conjectured above that $\sigma(\Pcal_{K}^{0})$ is equal to $\kappa$. This lead us to assume~(\ref{eq:densityPK}), and then to prove Conjecture~\ref{conj:main} above for the present random model. However, the authors of~\cite{Broderick:2011fk} formulated a weaker conjecture than~(\ref{eq:conjBFR1}) for which they have even stronger evidence, namely, there exists a real $\varsigma<2$ such that
\[
\#(\Pcal_{K}^{0}\cap\Pcal^{j})={\rm O}(2^{\varsigma j}) \qquad\text{as}\qquad j\to\infty,
\]
see~\cite[Conjecture~2]{Broderick:2011fk}. The last bound would readily imply that $\sigma(\Pcal_{K}^{0})$ is between $\kappa$ and $\kappa\varsigma$. This entices us to relax~(\ref{eq:densityPK}) by just assuming that the set $\Pcal_{K}$ satisfies
\[
\kappa\leq\sigma(\Pcal_{K})<2\kappa.
\]
Inspecting the above proofs, it is easy to see that Theorem~\ref{thm:conjmainrand}(\ref{thm:conjmainrand1}) still holds under this weaker assumption, whereas Theorem~\ref{thm:conjmainrand}(\ref{thm:conjmainrand2}) has to be replaced by the following statement: for any $\mu\geq 2$,
\[
\Hdim(\Mcal^{\bullet}(\mu)\cap K)=\max\left\{ \frac{2}{\mu} +  \kappa - 1, \frac{\sigma(\Pcal_{K})}{\mu} \right\}.
\]
In particular, the validity of Conjecture~\ref{conj:main} for the random model is equivalent to that of~(\ref{eq:densityPK}).

%%%%%%%%%%%%%%%%%%%%%%%%%%
\subsection{A probabilistic counterpart of the exponents $v_{b}$ and its connections with Conjecture~\ref{conj:vbnot3}}\label{subsec:probvb}
%%%%%%%%%%%%%%%%%%%%%%%%%%

Let us now modify the preceding ideas in order to put forward a randomized situation that is adapted to the exponents $v_{b}$. This way, we shall derive an analog of the dimension result~(\ref{eq:dimKv3geq}) of Levesley, Salp and Velani~\cite{Levesley:2007pd} when $b$ is a power of three, and give some probabilistic arguments leading to Conjecture~\ref{conj:vbnot3} otherwise.

For any integer $j\geq 1$, let $\Kcal^{j}$ denote the set of all integers $k\in\{0,\ldots,b^{j}-1\}$ such that $\gcd(k,b^{j})=1$. Furthermore, let $\Kcal_{K}^{j}$ denote the set of all integers $k\in\Kcal^{j}$ for which the rational number $kb^{-j}$ is in the Cantor set $K$. Now, given $j\geq 1$ and $k\in\Kcal^{j}$, we consider a random point $Y_{j,k}$ that is distributed uniformly in $K$ when $k\in\Kcal_{K}^{j}$, and uniformly in $\T$ otherwise. We also assume that the points $Y_{j,k}$ are independently distributed. This enables us to introduce the sets
\[
\Fcal_{b}(v)=\bigl\{\xi\in\T \:\bigl|\: d(\xi,Y_{j,k}) < b^{-(v+1)j} \text{~for i.m.~} j\geq 1 \text{~and~} k\in\Kcal^{j} \bigr\},
\]
as well as, for any point $\xi\in\T$, the exponent
\[
v^{\bullet}_{b}(\xi)=\sup\{v\in\R \:|\: \xi\in\Fcal_{b}(v)\},
\]
which may be seen as a probabilistic counterpart of the exponent $v_{b}$. Thus, in this randomized setting, the analogs of the sets $\Vcal_{b}(v)$ defined by~(\ref{eq:defVbv}) are merely the sets
\[
\Vcal^{\bullet}_{b}(v)=\{\xi\in\R \:|\: v^{\bullet}_{b}(\xi)\geq v \}.
\]

%%%%%%%%%%%%%%%%%%%%%%%%%%
\subsubsection{Case where $b$ is a power of three}
%%%%%%%%%%%%%%%%%%%%%%%%%%

In that situation, for any integer $j\geq 1$, there are $2b^{j}/3$ integers in the set $\Kcal^{j}$. Moreover, there are exactly $b^{\kappa j}$ rational numbers with reduced denominator $b^{j}$ in the Cantor set $K$. In other words, the set $\Kcal_{K}^{j}$ has cardinality $b^{\kappa j}$. Making the obvious changes to the proof of Theorem~\ref{thm:conjmainrand}, we easily deduce the following statement.

\begin{thm}
Let $b$ be a power of three. The following holds with probability one:
\begin{enumerate}
\item for any point $\xi\in K$,
\[
v^{\bullet}_{b}(\xi)\geq 0\,;
\]
\item for any real number $v\geq 0$,
\[
\Hdim(\Vcal^{\bullet}_{b}(v)\cap K)=\frac{\kappa}{v+1}.
\]
\end{enumerate}
\end{thm}

When $b=3$, we thus recover the same formula for the Hausdorff dimension as in the original context of the approximation by the triadic rational numbers, that is, the formula for the mere exponent $v_{3}$, see~(\ref{eq:dimKv3geq}).

%%%%%%%%%%%%%%%%%%%%%%%%%%
\subsubsection{Case where $b$ is not a power of three}
%%%%%%%%%%%%%%%%%%%%%%%%%%

Here, the cardinality of the set $\Kcal^{j}$ is again of the order of $b^{j}$; specifically, it is equal to $b^{j}$ times the product of $1-1/p$ when $p$ ranges over the prime factors of $b$. However, we do not know the cardinality of the set $\Kcal_{K}^{j}$ anymore. It is believed that the base $b$ representation is essentially independent of that in base three, on which the construction of $K$ heavily relies; this entices us to make the following conjecture: for all $\eps>0$,
\begin{equation}\label{eq:conjbaseb}
\#\Kcal_{K}^{j}={\rm O}(2^{\eps j}) \qquad\text{as}\qquad j\to\infty.
\end{equation}
Assuming that~(\ref{eq:conjbaseb}) holds, and adapting the proof of Theorem~\ref{thm:conjmainrand}, we then infer that almost surely, for any point $\xi\in K$, the exponent $v^{\bullet}_{b}(\xi)$ is nonnegative. Moreover, with probability one, for any real $v$,
\begin{equation}\label{eq:dimvbbulletKlo}
0\leq v\leq\frac{\kappa}{1-\kappa} \qquad\Longrightarrow\qquad \Hdim(\Vcal^{\bullet}_{b}(v)\cap K)=\frac{1}{v+1}+\kappa-1.
\end{equation}
and 
\begin{equation}\label{eq:dimvbbulletKhi}
v>\frac{\kappa}{1-\kappa} \qquad\Longrightarrow\qquad \Hdim(\Vcal^{\bullet}_{b}(v)\cap K)\leq 0.
\end{equation}
In particular, with probability one, the set of points $\xi$ in $K$ for which the exponent $v^{\bullet}_{b}(\xi)$ is at least equal to $\kappa/(1-\kappa)$ is nonempty. Moreover, the discussion that precedes Proposition~\ref{prp:approxunifT} actually implies that this set is dense in $K$.

The above approach does not enable us to determine whether or not the dimension in~(\ref{eq:dimvbbulletKhi}) is equal to zero, that is, whether or not there exists a point $\xi\in K$ such that $v^{\bullet}_{b}(\xi)\geq v$, when $v$ is larger than $\kappa/(1-\kappa)$. However, a straightforward adaptation of the proof of Theorem~\ref{thm:conjmainrand} implies that
\[
\sum_{j=1}^{\infty}\#\Kcal_{K}^{j}<\infty \qquad \Longrightarrow \qquad \as \quad \forall\xi\in K \quad v^{\bullet}_{b}(\xi)\leq\frac{\kappa}{1-\kappa}.
\]
Thus, under a much stronger assumption than~(\ref{eq:conjbaseb}), our method shows that the dimension in~(\ref{eq:dimvbbulletKhi}) is equal to $-\infty$. In any case, deciding whether or not the dimension is zero in~(\ref{eq:dimvbbulletKhi}) certainly requires much more information on the sets $\Kcal_{K}^{j}$ than the mere~(\ref{eq:conjbaseb}).

As regards the original exponent $v_{b}$, we suspect that the dimension in~(\ref{eq:dimvbbulletKhi}) is equal to $-\infty$, meaning that $v_{b}$ is bounded from above by $\kappa/(1-\kappa)$ on the Cantor set $K$. Combined with~(\ref{eq:dimvbbulletKlo}), this is what lead us to Conjecture~\ref{conj:vbnot3} above.

%%%%%%%%%%%%%%%%%%%%%%%%%%
%%%%%%%%%%%%%%%%%%%%%%%%%%
\section{Approximation by uniform random points: general results}\label{sec:randomgene}
%%%%%%%%%%%%%%%%%%%%%%%%%%
%%%%%%%%%%%%%%%%%%%%%%%%%%

The purpose of this section is to study the situation in which the sequence of approximating points is chosen at random. For convenience, we work again on the circle $\T=\R/\Z$, endowed with the usual quotient distance $\dist$. Given a sequence $\Xseq=(X_{n})_{n \geq 1}$ of random variables in the circle $\T$ and a sequence $\rseq=(r_{n})_{n\geq 1}$ of real numbers in $(0,1]$, let us consider the random subset $\Ecal(\Xseq,\rseq)$ of $\T$ defined by
\begin{equation}\label{eq:defEcalXr}
\Ecal(\Xseq,\rseq) = \left\{\xi\in\T \:\bigl|\: d(\xi,X_{n}) < r_{n} \text{~for i.m.~} n\geq 1 \right\},
\end{equation}
and formed by the points that are approximated at a rate given by $r_{n}$ by the random points $X_{n}$. Our purpose is to study the probability with which the random set $\Ecal(\Xseq,\rseq)$ intersects a given nonempty compact set $G\subseteq\T$, and to describe the size properties of the intersection in the situation where it is nonempty. Such a description will be obtained by studying the value of the Hausdorff measures of the intersection $\Ecal(\Xseq,\rseq)\cap G$. Throughout this section, we assume that the random points $X_{n}$ are chosen according to the Lebesgue measure $\leb$ on the circle $\T$.

%%%%%%%%%%%%%%%%%%%%%%%%%%
\subsection{Size of the intersection with a compact set: upper bounds}
%%%%%%%%%%%%%%%%%%%%%%%%%%

By now, we do not make any assumption on the correlations between these random points. Our first result gives an upper bound on the size of the intersection of the random set $\Ecal(\Xseq,\rseq)$ with a fixed compact set $G$ whose size is controlled in terms of the finiteness of certain packing premeasures. We refer to Section~\ref{sec:proofthmmajsizegene} for its proof.

\begin{thm}\label{thm:majsizegene}
Let $G$ denote a nonempty compact subset of the circle, and let $g$ be a doubling gauge function such that $\prepack^{g}(G)<\infty$.
\begin{enumerate}
\item\label{thm:majsizegene1} The following holds:
\[
\sum_{n=1}^{\infty} \frac{r_{n}}{g(r_{n})}<\infty \qquad\Longrightarrow\qquad \as \quad \Ecal(\Xseq,\rseq)\cap G=\emptyset.
\]
\item\label{thm:majsizegene2} For any doubling gauge function $h$, the following holds:
\[
\sum_{n=1}^{\infty} \frac{h(r_{n})r_{n}}{g(r_{n})}<\infty \qquad\Longrightarrow\qquad \as \quad \hau^{h}(\Ecal(\Xseq,\rseq)\cap G)=0.
\]
\end{enumerate}
\end{thm}

%%%%%%%%%%%%%%%%%%%%%%%%%%
\subsection{Size of the intersection with a compact set: lower bounds under a weak dependence condition}\label{subsec:randomgenelow}
%%%%%%%%%%%%%%%%%%%%%%%%%%

Our purpose is now to give a converse to Theorem~\ref{thm:majsizegene}(\ref{thm:majsizegene1}), under the assumption that the random approximating points $X_{n}$ are independent or somewhat close to being so, in the following sense. First, for any sequence $\Bseq=(B_{n})_{n\geq 1}$ of Borel subsets of the circle $\T$ 
with positive Lebesgue measure, let us set
\[
\theta(\Xseq,\Bseq)=\sup_{n\geq 1}\frac{\prob(X_{1}\in B_{1},\ldots,X_{n}\in B_{n})}{\prob(X_{1}\in B_{1})\cdot\ldots\cdot\prob(X_{n}\in B_{n})}.
\]
Plainly, $\theta(\Xseq,\Bseq)$ is always at least one, and is equal to one regardless of the choice of the sequence $\Bseq$ when the variables $X_{n}$ are independent. A way of relaxing the independence assumption is then to impose a control on the maximal ratios $\theta(\Xseq,\Bseq)$ by supposing that
\begin{equation}\label{eq:defThetaXB}
\Theta(\Xseq,\Bcal)=\sup_{\Bseq\in\Bcal}\theta(\Xseq,\Bseq)<\infty,
\end{equation}
where $\Bcal$ denotes an appropriately chosen collection of sequences of Borel sets. Note that the above condition is more stringent as $\Bcal$ becomes larger.

The collection on which we shall impose a control is denoted by $\Bcal(\rseq)$ and is defined as follows in terms of the sequence $\rseq$ that gives the approximation rate. For any $k\in\{0,\ldots,q_{n}-1\}$, let $I_{n,k}$ denote the image of the interval $[k/q_{n},(k+1)/q_{n})$ under the projection modulo one. Here, we choose $q_{n}$ to be equal to $\lceil 1/r_{n}\rceil$, where $\lceil\,\cdot\,\rceil$ stands for the ceiling function. Then, $\Bcal(\rseq)$ is defined as the collection of all sequences of the form $(I_{n,k_{n}})_{n\geq 1}$, where $(k_{n})_{n\geq 1}$ is a sequence of nonnegative integers satisfying $k_{n}<q_{n}$.

Our assumption on the joint law of the points $X_{n}$ is finally that $\Theta(\Xseq,\Bcal(\rseq))$ is finite. When the points $X_{n}$ are independent, this condition is clearly satisfied because $\Theta(\Xseq,\Bcal(\rseq))$ is then equal to one regardless of the choice of the sequence $\rseq$. Furthermore, if the series $\sum_{n}r_{n}$ converges, then the above finiteness assumption is equivalent to the existence of a positive real number $C$ such that
\[
\prob\left(\bigcap_{n=1}^{v} \{X_{n}\in I_{n,k_{n}}\}\right)\leq C \prod_{n=1}^{v} r_{n}, 
\]
for any integer $v\geq 1$ and any choice of the integers $k_{n}\in\{0,\ldots,q_{n}-1\}$, because the points $X_{n}$ are uniformly distributed. Our converse to Theorem~\ref{thm:majsizegene}(\ref{thm:majsizegene1}) is now the following result, which is proven in Section~\ref{sec:proofthmnonemptygene}.

\begin{thm}\label{thm:nonemptygene}
Let $G$ denote a nonempty compact subset of the circle $\T$, and let $g$ be a gauge function such that $\hau^{g}(G)>0$. Then,
\[
\left\{\begin{array}{l}
\Theta(\Xseq,\Bcal(\rseq))<\infty\\[1mm]
\sum_{n} r_{n}/g(r_{n})=\infty
\end{array}\right.
\qquad\Longrightarrow\qquad \as \quad \Ecal(\Xseq,\rseq)\cap G\neq\emptyset.
\]
\end{thm}

The last result of this section gives a partial converse to Theorem~\ref{thm:majsizegene}(\ref{thm:majsizegene2}), under the same assumption as in the preceding result, {\em i.e.}~the finiteness of $\Theta(\Xseq,\Bcal(\rseq))$. Before stating this result, let us consider a compact subset $G$ of the circle with positive Hausdorff $g$-measure for a given gauge function $g$. The set $\Ecal(\Xseq,\rseq)\cap G$ is clearly smaller than $G$. Thus, in order to describe the size properties of this set in terms of generalized Hausdorff measures, we may restrict our attention to the gauge functions $h$ that increase faster than $g$ in the sense that $g/h$ monotonically tends to zero at the origin. An expected converse to Theorem~\ref{thm:majsizegene}(\ref{thm:majsizegene2}) is then the following:
\begin{equation}\label{eq:idealminsizegene}
\sum_{n=1}^{\infty} \frac{h(r_{n})r_{n}}{g(r_{n})}=\infty \qquad\Longrightarrow\qquad \as \quad \hau^{h}(\Ecal(\Xseq,\rseq)\cap G)>0.
\end{equation}
Theorem~\ref{thm:minsizegene} below gives a slightly weaker form of this statement. In fact, we make the additional assumption that $h$ increases faster than $g$, with respect to a third gauge function $\ph$ which is used as a proxy for $g/h$ in the divergence condition above. To be more precise, given three gauge functions $g$, $h$ and $\ph$, we say that $h$ increases $\ph$-faster than $g$, and we write $h\prec_{\ph} g$, when $g/h$ monotonically tends to zero at the origin and satisfies
\begin{equation}\label{eq:defprecph}
\sum_{j=1}^{\infty} \frac{g(2^{-j})}{h(2^{-j})}\left(\frac{1}{\ph(2^{-j})}-\frac{1}{\ph(2^{-(j-1)})}\right)<\infty.
\end{equation}
In that case, note that $g/h$ coincides with a gauge function near zero.

\begin{thm}\label{thm:minsizegene}
Let $G$ denote a nonempty compact subset of the circle $\T$, and let $g$ be a gauge function such that $\hau^{g}(G)>0$. Then, for any gauge functions $h$ and $\ph$ such that $h\prec_{\ph} g$,
\[
\left\{\begin{array}{l}
\Theta(\Xseq,\Bcal(\rseq))<\infty\\[1mm]
\sum_{n} r_{n}/\ph(r_{n})=\infty
\end{array}\right.
\qquad\Longrightarrow\qquad \as \quad \hau^{h}(\Ecal(\Xseq,\rseq)\cap G)>0.
\]
\end{thm}

Theorem~\ref{thm:minsizegene} follows straightforwardly from Theorem~\ref{thm:nonemptygene} with the help of Lemma~\ref{lem:codim} below. Indeed, in view of this lemma, it suffices to show that $\Ecal(\Xseq,\rseq)$ intersects every compact subset of the circle with positive Hausdorff $\ph$-measure, a fact that follows from Theorem~\ref{thm:nonemptygene}, since $\Theta(\Xseq,\Bcal(\rseq))$ is finite and $\sum_{n} r_{n}/\ph(r_{n})$ diverges.

\begin{lem}\label{lem:codim}
Let us consider a random subset $E$ of the circle and let us assume that there exists a gauge function $\ph$ such that for any compact set $G\subseteq\T$,
\[
\hau^{\ph}(G)>0 \qquad\Longrightarrow\qquad \as \quad E\cap G\neq\emptyset.
\]
Then, for any compact set $G\subseteq\T$ and any gauge function $g$,
\[
\hau^{g}(G)>0 \qquad\Longrightarrow\qquad  \forall h\prec_{\ph} g \quad \as \quad \hau^{h}(E\cap G)>0.
\]
\end{lem}

Lemma~\ref{lem:codim} can be seen as an extension of~\cite[Lemma~3.4]{Khoshnevisan:2000fj} to general Hausdorff measures. Its proof, given in Section~\ref{subsec:prooflemcodim}, relies on the remarkable properties satisfied by a family of compact sets obtained via a variant of Mandelbrot's fractal percolation process that we introduce and study in Section~\ref{sec:fracperc}.

Let us point out a very simple situation in which the condition $h\prec_{\ph} g$ defined by~(\ref{eq:defprecph}) is satisfied: it suffices to assume that the gauge function $h$ increases faster than $g$ in the sense that $h/g$ is monotonic near zero and satisfies
\[
\int_{0}^{1}\frac{h(r)}{g(r)}\,\pi(\dd r)=\infty
\]
for some probability measure $\pi$ on $(0,1]$. Note that the function $h$ may nevertheless be very close to $g$ near zero because the probability measure $\pi$ may well concentrate its mass near this point: for instance, if $h(r)=g(r)(\log^{\circ p}(1/r))^{\eps}$ for some $\eps>0$ and $p\geq 1$, where $\log^{\circ p}$ denotes the $p$-th iterate of the logarithm, then the gauge functions verify the above condition. Now, it is straightforward to check that~(\ref{eq:defprecph}) holds if $\ph$ is a gauge function such that
\[
\frac{1}{\ph(s)}-\frac{1}{\ph(1)}=\int_{r\in(s,1]}\frac{h(r)}{g(r)}\,\pi(\dd r), 
\]
for all $s\in(0,1]$. Moreover, when $\ph$ is chosen as above, the Fubini-Tonelli theorem ensures that
\[
\sum_{n=1}^{\infty}\frac{r_{n}}{\ph(r_{n})}\geq\int_{0}^{1}\frac{h(r)}{g(r)}\left(\sum_{n=1}^{\infty}r_{n}\ind_{\{r_{n}<r\}}\right)\,\pi(\dd r).
\]
As a consequence, assuming the finiteness of $\Theta(\Xseq,\Bcal(\rseq))$ and considering a compact set $G$ with positive Hausdorff $g$-measure, Theorem~\ref{thm:minsizegene} implies that
\[
\int_{0}^{1}\frac{h(r)}{g(r)}\left(\sum_{n=1}^{\infty}r_{n}\ind_{\{r_{n}<r\}}\right)\,\pi(\dd r)=\infty \quad\Longrightarrow\quad \as \quad \hau^{h}(\Ecal(\Xseq,\rseq)\cap G)>0.
\]
Note that the divergence of this integral implies that of the series arising in the statement of Theorem~\ref{thm:majsizegene}(\ref{thm:majsizegene2}). The above result is therefore slightly weaker than the expected converse~(\ref{eq:idealminsizegene}) to this theorem. To be specific, our approach leaves open the case in which the series diverges, but the above integral is convergent for every possible choice of the probability measure $\pi$.

%%%%%%%%%%%%%%%%%%%%%%%%%%
\section{Application to the approximation by independent points}\label{sec:approxindep}
%%%%%%%%%%%%%%%%%%%%%%%%%%

%%%%%%%%%%%%%%%%%%%%%%%%%%
\subsection{Uniform distribution in the circle and intersection with a regular set}
%%%%%%%%%%%%%%%%%%%%%%%%%%

We keep on supposing that the approximating points $X_{n}$ are uniformly distributed in the circle $\T$, that is, are chosen according to the Lebesgue measure $\leb$. In addition, we assume that these points are independent random variables. In particular, the maximal ratios $\Theta(\Xseq,\Bcal)$ defined by~(\ref{eq:defThetaXB}) are equal to one, and the weak dependence assumption that we made in order to derive the lower bounds in the previous section is plainly satisfied. As a consequence, all the results stated in Section~\ref{sec:randomgene} apply in the present setting.

Our purpose is now to deduce from these results simpler statements that only involve Hausdorff dimensions and a probabilistic analog of the irrationality exponent that is defined as follows. To proceed, let us make two additional assumptions on the sequence $\rseq=(r_{n})_{n\geq 1}$ of approximation radii. First, since the joint law of the approximating points $X_{n}$ is invariant under rearrangement, there is no loss of generality in assuming that the sequence $\rseq$ is nonincreasing. Second, we suppose that there exists a critical value $\rho\in(0,\infty)$ such that
\begin{equation}\label{eq:sumrnnu}
\left\{\begin{array}{rcl}
\nu<\rho & \Longrightarrow & \sum_{n}r_{n}^{\nu}=\infty \\[2mm]
\nu>\rho & \Longrightarrow & \sum_{n}r_{n}^{\nu}<\infty.
\end{array}\right.
\end{equation}
For any real number $\nu\geq 0$, let $\rseq^{\nu}$ denote the sequence formed by the numbers $r_{n}^{\nu}$, so that $\Ecal(\Xseq,\rseq^{\nu})$ is the set obtained by replacing $r_{n}$ by $r_{n}^{\nu}$ in the definition~(\ref{eq:defEcalXr}) of the set $\Ecal(\Xseq,\rseq)$. Observe that the mapping $\nu\mapsto\Ecal(\Xseq,\rseq^{\nu})$ is nonincreasing. Therefore, for any point $\xi\in\T$, we may define
\[
\nu_{\Xseq,\rseq}(\xi) = \sup\{\nu\geq 0 \:|\: \xi\in\Ecal(\Xseq,\rseq^{\nu})\}\,;
\]
this may be seen as the analog of the irrationality exponent for the approximation by the random points $X_{n}$ with the rates $r_{n}$.

Given that the points $X_{n}$ are independently and uniformly distributed, we may apply Shepp's theorem on Dvoretzky's covering problem~\cite{Shepp:1972qj}, thereby inferring that with probability one, the set $\Ecal(\Xseq,\rseq^{\nu})$ coincides with the whole circle when $\nu$
is smaller than the critical value $\rho$ defined by~(\ref{eq:sumrnnu}). As a result,
\begin{equation}\label{eq:coverunifT}
\as \quad \forall \xi\in\T \qquad \nu_{\Xseq,\rseq}(\xi)\geq\rho.
\end{equation}
Moreover, Corollary~1 in~\cite{Durand:2010fk} yields the value of the Hausdorff dimension of the set $\Ecal(\Xseq,\rseq^{\nu})$, specifically, with probability one, for all $\nu\geq\rho$,
\[
\Hdim\Ecal(\Xseq,\rseq^{\nu})=\frac{\rho}{\nu},
\]
from which it is straightforward to deduce that
\begin{equation}\label{eq:dimnuXrgeq}
\as \quad \forall \nu\geq\rho \qquad \Hdim\{\xi\in\T \:|\: \nu_{\Xseq,\rseq}(\xi)\geq\nu\}=\frac{\rho}{\nu}.
\end{equation}

In order to make our statements even more concrete, we further assume that the compact set $G$ with which the intersections are taken is Ahlfors regular with dimension $\gamma$ in $(0,1]$, see Definition~\ref{df:Ahlfors}. Applying Theorems~\ref{thm:majsizegene} and~\ref{thm:nonemptygene}, we infer that
\begin{equation}\label{eq:hitprobEXrnu}
\prob(\Ecal(\Xseq,\rseq^{\nu})\cap G\neq\emptyset)=
\begin{cases}
1 & \mbox{if } (1-\gamma)\nu<\rho \\[2mm]
0 & \mbox{if } (1-\gamma)\nu>\rho.
\end{cases}
\end{equation}
Moreover, Theorems~\ref{thm:majsizegene} and~\ref{thm:minsizegene} ensure that if $(1-\gamma)\nu<\rho$, then with probability one,
\begin{equation}\label{eq:dimhitEXrnu}
\Hdim(\Ecal(\Xseq,\rseq^{\nu})\cap G)=\frac{\rho}{\nu}+\gamma-1.
\end{equation}
Here, we recover two results obtained recently by Li, Shieh and Xiao in~\cite{Li:2012uq}. More precisely, building on the study of the limsup random fractals performed in~\cite{Khoshnevisan:2000fj}, these authors computed the hitting probabilities of the random set $\Ecal(\Xseq,\rseq^{\nu})$, and the Hausdorff and packing dimensions of its intersection with a fixed analytic set, see Theorem~2.1 and Corollary~2.5 in~\cite{Li:2012uq}.

When $(1-\gamma)\nu\neq\rho$, we straightforwardly deduce that the right-hand side of~(\ref{eq:hitprobEXrnu}) gives the probability that $\nu_{\Xseq,\rseq}(\xi)\geq\nu$ for some point $\xi\in G$. The critical case where $(1-\gamma)\nu=\rho$ does not explicitly follows from either Theorems~\ref{thm:majsizegene} and~\ref{thm:nonemptygene} above or the results of~\cite{Li:2012uq}. However, inspecting the proof of Theorem~\ref{thm:nonemptygene}, we see that the sets $\Ecal(\Xseq,\rseq^{\nu-\eps})\cap G$, for $\eps>0$, are almost surely dense in the complete metric space $G$. Taking the intersection of these sets over a sequence $(\eps_{n})_{n\geq 1}$ converging to zero, and applying the Baire category theorem, we deduce that the set of points $\xi\in G$ such that $\nu_{\Xseq,\rseq}(\xi)\geq\nu$ is almost surely dense in $G$ as well, thereby being nonempty. Furthermore, one easily checks that the right-hand side of~(\ref{eq:dimhitEXrnu}) also gives the Hausdorff dimension of these set of points. Thus, we end up with the next statement.

\begin{prp}\label{prp:approxunifT}
Let $G$ be a compact subset of the circle $\T$, and assume that $G$ is regular with dimension $\gamma\in(0,1]$. If the variables $X_{n}$ are independently and uniformly distributed in $\T$, then
\[
\as \quad \forall \xi\in G \qquad \rho\leq\nu_{\Xseq,\rseq}(\xi)\leq\frac{\rho}{1-\gamma}.
\]
Moreover, for any real number $\nu\geq\rho$ such that $(1-\gamma)\nu\leq\rho$,
\[
\as \qquad \Hdim\{ \xi\in G \:|\: \nu_{\Xseq,\rseq}(\xi)\geq\nu \}=\frac{\rho}{\nu}+\gamma-1.
\]
\end{prp}

The above result shows that the maximal rate at which the points of a regular set may be approximated by a sequence of independently and uniformly distributed points is directly controlled by the size of the set; indeed, the value of $\gamma$ induces a specific limitation on the rate with which the points in $G$ may be approximated.

In addition, combined with~(\ref{eq:dimnuXrgeq}), the previous result ensures that if $(1-\gamma)\nu\leq\rho$, then the Hausdorff codimension of the intersection of the set of all $\xi\in\T$ with $\nu_{\Xseq,\rseq}(\xi)\geq\nu$ and the set $G$ is the sum of their codimensions. Such a behavior is expected to be somewhat generic and is in stark contrast with the special situation of sets with large intersection (sometimes also termed as intersective sets) where the Hausdorff dimension of the intersection of the sets is equal to the minimum of their dimensions; we refer to Chapter~8 in~\cite{Falconer:2003oj}, and to~\cite{Bugeaud:2004wc,Durand:2007uq,Falconer:1994hx} for details. Let us mention here that the set of all $\xi\in\T$ with $\nu_{\Xseq,\rseq}(\xi)\geq\nu$ is known to be almost surely intersective, as a consequence of Theorem~2 in~\cite{Durand:2010fk}. In addition, when $\gamma<1$, the set $G$ cannot be intersective (because an intersective subset of $\R$ has packing dimension equal to one, see~\cite{Falconer:1994hx}), and this is consistent with the observation that
\begin{align*}
\Hdim\{ \xi\in G \:|\: \nu_{\Xseq,\rseq}(\xi)\geq\nu \}
&=\Hdim\{\xi\in\T \:|\: \nu_{\Xseq,\rseq}(\xi)\geq\nu\}+\Hdim G-1\\
&<\min\left\{\Hdim\{\xi\in\T \:|\: \nu_{\Xseq,\rseq}(\xi)\geq\nu\},\Hdim G\right\}
\end{align*}
with probability one, under the further assumption that $\nu>\rho$.

%%%%%%%%%%%%%%%%%%%%%%%%%%
\subsection{Uniform distribution in a regular set}
%%%%%%%%%%%%%%%%%%%%%%%%%%

We now suppose that the variables $X_{n}$ are uniformly distributed in a given compact subset $G$ of the circle that is assumed to be regular with dimension $\gamma\in(0,1]$. The common law of the random variables $X_{n}$ is thus the normalized $\gamma$-dimensional Hausdorff measure restricted to $G$. It is clear that the sets $\Ecal(\Xseq,\rseq^{\nu})$ are contained in $G$, so that $\nu_{\Xseq,\rseq}(\xi)=0$ when the point $\xi$ does not belong to $G$. In addition, we have the following lower bound on $\nu_{\Xseq,\rseq}(\xi)$ when the point $\xi$ is in $G$:
\[
\as \quad \forall \xi\in G \qquad \nu_{\Xseq,\rseq}(\xi)\geq\frac{\rho}{\gamma}.
\]
This bound generalizes~(\ref{eq:coverunifT}) and follows directly from the next lemma.

\begin{lem}\label{lem:EXrcoverG}
If $\gamma\nu<\rho$, then $\Ecal(\Xseq,\rseq^{\nu})=G$ with probability one.
\end{lem}

\begin{proof}
Let $\eps$ denote a positive real number with $\gamma\nu(1+\eps)<\rho$. In view of~(\ref{eq:sumrnnu}), there exists an infinite set $\Ncal$ of integers $n$ such that $r_{n}\geq n^{-(1+\eps)/\rho}$. Given $n\in\Ncal$, let us now consider a collection of points $\xi_{1},\ldots,\xi_{u_{n}}$ in $G$ such that the arcs $\opball{\xi_{j}}{n^{-1/\gamma}/2}$ are disjoint, and assume that $u_{n}$ is maximal for this property. The arcs $\opball{\xi_{j}}{n^{-1/\gamma}}$ then cover $G$, so that
\[
G\not\subseteq\bigcup_{i=1}^{n}\opball{X_{i}}{r_{i}^{\nu}}
\qquad\Longrightarrow\qquad
\exists j \quad \xi_{j}\not\in\bigcup_{i=1}^{n}\opball{X_{i}}{r_{n}'},
\]
where $r_{n}'=r_{n}^{\nu}-n^{-1/\gamma}$. Since the random points $X_{i}$ are independent, this yields
\begin{equation}\label{eq:upbndprobGcover}
\prob\left(G\not\subseteq\bigcup_{i=1}^{n}\opball{X_{i}}{r_{i}^{\nu}}\right)
\leq \sum_{j=1}^{u_{n}}\prod_{i=1}^{n}\left(1-\prob(\xi_{j}\in\opball{X_{i}}{r_{n}'})\right).
\end{equation}
Moreover, the points $X_{i}$ are uniformly distributed in the regular set $G$, so that
\[
\prob(\xi_{j}\in\opball{X_{i}}{r_{n}'})=
\frac{\hau^{\gamma}(G\cap\opball{\xi_{j}}{r_{n}')}}{\hau^{\gamma}(G)}\geq\frac{(r_{n}')^{\gamma}}{c\hau^{\gamma}(G)}.
\]
The fact that $G$ is regular with dimension $\gamma$ also implies that $u_{n}\leq c' n$ for some constant $c'>0$. We deduce that the right-hand side of~(\ref{eq:upbndprobGcover}) is bounded from above by
\[
c' n\exp\left(-\frac{n(r_{n}')^{\gamma}}{c\hau^{\gamma}(G)}\right)
\leq c' n\exp(-c'' n^{1-\gamma\nu(1+\eps)/\rho})
\]
for some other constant $c''>0$. Finally, letting $n$ tend to infinity through $\Ncal$, we deduce that
\[
\prob\left(G\not\subseteq\bigcup_{i=1}^{\infty}\opball{X_{i}}{r_{i}^{\nu}}\right)=0.
\]
In other words, the arcs $\opball{X_{i}}{r_{i}^{\nu}}$, for $i\geq 1$, cover the set $G$ with probability one. For any fixed $i_{0}\geq 1$, we can obviously reproduce the same reasoning when only considering the arcs indexed by $i\geq i_{0}$, thereby obtaining that these arcs also suffice to cover $G$ almost surely. The result follows.
\end{proof}

The case in which $\gamma\nu\geq\rho$ is not covered by the previous result, and it is then natural to ask for the size of the set $\Ecal(\Xseq,\rseq^{\nu})$. The purpose of the next statement is to give a simple answer to this question.

\begin{lem}
If $\gamma\nu\geq\rho$, then $\Hdim\Ecal(\Xseq,\rseq^{\nu})=\rho/\nu$ with probability one.
\end{lem}

\begin{proof}
The upper bound follows from the obvious fact that the set $\Ecal(\Xseq,\rseq^{\nu})$ is covered by the arcs $\opball{X_{n}}{r_{n}^{\nu}}$, for $n$ larger than any given integer. To be specific, for $\eps>0$ and $n_{0}\in\N$ such that $2r_{n_{0}}^{\nu}<\eps$, it is clear that $\hau^{s}_{\eps}(\Ecal(\Xseq,\rseq^{\nu}))$ is bounded from above by $\sum_{n\geq n_{0}} (2r_{n}^{\nu})^{s}$. In view of~(\ref{eq:sumrnnu}), this series converges for $\nu s>\rho$. Letting $n_{0}$ tend to infinity and $\eps$ go to zero, we then deduce that $\hau^{s}(\Ecal(\Xseq,\rseq^{\nu}))=0$.

In order to prove the lower bound, let us consider a positive real number $s$ such that $\nu s<\rho$. By virtue of Lemma~\ref{lem:EXrcoverG}, the set $\Ecal(\Xseq,\rseq^{\nu s/\gamma})$ coincides with the whole set $G$ with probability one. Given that $G$ is regular with dimension $\gamma$, the general mass transference principle of Beresnevich and Velani then ensures that the set $\Ecal(\Xseq,\rseq^{\nu})$ has Hausdorff dimension at least $s$, see~\cite[Theorem~3]{Beresnevich:2005vn}.
\end{proof}

It is now straightforward to deduce from the previous lemma the following generalization of~(\ref{eq:coverunifT}) and~(\ref{eq:dimnuXrgeq}), which has to be compared with Proposition~\ref{prp:approxunifT}.

\begin{prp}\label{prp:approxunifGregular}
Let $G$ be a compact subset of the circle $\T$, and assume that $G$ is regular with dimension $\gamma\in(0,1]$. If the variables $X_{n}$ are independently and uniformly distributed in $G$, then
\[
\as \quad \forall \xi\in G \qquad \nu_{\Xseq,\rseq}(\xi)\geq\frac{\rho}{\gamma}.
\]
Moreover, for any real number $\nu\geq\rho/\gamma$,
\[
\as \qquad \Hdim\{\xi\in G \:|\: \nu_{\Xseq,\rseq}(\xi)\geq\nu\}=\frac{\rho}{\nu}.
\]
\end{prp}

%%%%%%%%%%%%%%%%%%%%%%%%%%
%%%%%%%%%%%%%%%%%%%%%%%%%%
\section{Application to the approximation by fractional parts}\label{sec:approxfrac}
%%%%%%%%%%%%%%%%%%%%%%%%%%
%%%%%%%%%%%%%%%%%%%%%%%%%%

Let $\{\,\cdot\,\}$ stand for the fractional part function. Identifying the circle $\T$ with the interval $[0,1)$, we may also regard the mapping $x\mapsto\{x\}$ as the projection modulo one from $\R$ onto $\T$. The purpose of this section is to apply the general results stated in Section~\ref{sec:randomgene} to the situation where the approximating points $X_{n}$ are of the form $\{a_{n}X\}$, where $\aseq=(a_{n})_{n\geq 1}$ is a sequence of positive integers and $X$ is a point chosen uniformly in the interval $[0,1)$, that is, according to the Lebesgue measure thereon. It is easy to see that we match the general framework of Section~\ref{sec:randomgene}: the random points $X_{n}=\{a_{n}X\}$ are clearly distributed according to the Lebesgue measure on the circle. This is due to the well known fact that, for any integer $m\geq 1$, the transformation $x\mapsto\{m x\}$ preserves the Lebesgue measure on $[0,1)$.

Given a sequence $\rseq=(r_{n})_{n\geq 1}$ of real numbers in $(0,1]$, the random subset of the circle defined by~(\ref{eq:defEcalXr}) is now of the form
\[
\Gcal(\aseq,\rseq) = \left\{\xi\in\T \:\bigl|\: d(\xi,\{a_{n}X\})<r_{n} \text{~for i.m.~} n\geq 1 \right\}.
\]
All the hypotheses of Theorem~\ref{thm:majsizegene} are fulfilled, so we may directly apply this result to the above set $\Gcal(\aseq,\rseq)$. In order to apply the other results of Section~\ref{sec:randomgene}, we need to show that the weak dependence condition is satisfied.

We shall show that if the sequence $\aseq=(a_{n})_{n\geq 1}$ grows sufficiently fast, then the random points $\{a_{n}X\}$ are close enough to being independent to ensure that all the results of Section~\ref{sec:randomgene} apply. Specifically, using the notations of Section~\ref{subsec:randomgenelow}, this amounts to showing the finiteness of $\Theta((\{a_{n}X\})_{n\geq 1},\Bcal(\rseq))$, when the integers $a_{n}$ grow fast enough. This is the purpose of the next result.

\begin{prp}\label{prp:angrowTheta}
For any sequence $\aseq=(a_{n})_{n\geq 1}$ of positive integers and any sequence $\rseq=(r_{n})_{n\geq 1}$ of real numbers in $(0,1]$,
\[
\Theta((\{a_{n}X\})_{n\geq 1},\Bcal(\rseq))\leq 3\exp\left(4\sum_{n=1}^{\infty}\frac{a_{n}}{r_{n}a_{n+1}}\right).
\]
\end{prp}

\begin{proof}
Making use of the notations of Section~\ref{subsec:randomgenelow}, let $q_{n}=\lceil 1/r_{n}\rceil$ and let $I_{n,k}$ denote the image of the interval $[k/q_{n},(k+1)/q_{n})$ under the projection onto the circle. We now have
\[
\Theta((\{a_{n}X\})_{n\geq 1},\Bcal(\rseq))=\sup_{\kseq}\sup_{v\geq 1}\prob\left(\bigcap_{n=1}^{v} \bigl\{\{a_{n}X\}\in I_{n,k_{n}}\bigr\}\right)\prod_{n=1}^{v} q_{n},
\]
where the outer supremum is taken over the set of sequences $\kseq=(k_{n})_{n\geq 1}$ of nonnegative integers less than $q_{n}$. The random variables $\{a_{n}X\}$ are uniformly distributed on the circle; the probability arising above is thus equal to
\[
\int_{0}^{1}\prod_{n=1}^{v}\ind_{[k_{n}/q_{n},(k_{n}+1)/q_{n})}(\{a_{n}x\})\,\dd x
\leq \left(1+\frac{2}{a_{1}}\right) \prod_{n=1}^{v}\left(\frac{1}{q_{n}}+\frac{2a_{n}}{a_{n+1}}\right).
\]
We conclude by remarking that $q_{n}\leq 2/r_{n}$ for any $n\geq 1$, and that $1+x\leq\ee^{x}$ for any real number $x$.
\end{proof}

Proposition~\ref{prp:angrowTheta} directly shows that when the sequence
of integers $(a_{n})_{n \geq 1}$ grows fast enough to ensure the convergence of the series $\sum_{n}a_{n}/(r_{n}a_{n+1})$, 
then the dependence between the random points $\{a_{n}X\}$ is sufficiently weak to guarantee that all the results of Section~\ref{sec:randomgene} are applicable. In that situation, this leads us to a rather precise description of the size of the intersection of the random set $\Gcal(\aseq,\rseq)$ with a fixed compact subset of the circle. By way of illustration, we shall now determine the Hausdorff dimension of such an intersection in the case where the compact set is regular, keeping in mind that the results of Section~\ref{sec:randomgene} actually yield much finer statements.

To this end, let $G$ denote a compact subset of the circle $\T$, and let us suppose that $G$ is regular with dimension $\gamma\in(0,1]$. Let us also assume that the sequence $\rseq$ satisfies~(\ref{eq:sumrnnu}) for some real $\rho>0$. The convergence of the aforementioned series is then guaranteed when
\begin{equation}\label{eq:condgrowthan}
\liminf_{n\to\infty}\frac{\log (a_{n}/a_{n+1})}{\log r_{n}}>1+\rho.
\end{equation}
In that situation, we may apply all the results of Section~\ref{sec:randomgene}. Therefore, when $\rho+\gamma<1$, Theorem~\ref{thm:majsizegene} ensures that the intersection $\Gcal(\aseq,\rseq)\cap G$ is almost surely empty. When $\rho+\gamma$ is equal to one, the intersection is empty with probability one or zero, according to the convergence or divergence of the series $\sum_{n} r_{n}^{\rho}$, respectively; this is due to Theorems~\ref{thm:majsizegene} and~\ref{thm:nonemptygene}. Finally, when $\rho+\gamma>1$, Theorem~\ref{thm:nonemptygene} implies that the intersection is almost surely nonempty; by virtue of Theorems~\ref{thm:majsizegene} and~\ref{thm:minsizegene}, its Hausdorff dimension then satisfies
\[
\as \qquad \Hdim\left(\Gcal(\aseq,\rseq)\cap G\right)=\min\{\rho,1\}+\gamma-1.
\]

With a view to establishing a connection with existing results from metric number theory, let us consider the particular case where the radii $r_{n}$ are of the form $n^{-\nu}$, where $\nu$ is a positive real number. The critical exponent coming into play in~(\ref{eq:sumrnnu}) is then given by $\rho=1/\nu$. Furthermore, the condition~(\ref{eq:condgrowthan}) is verified regardless of the value of $\nu$ when the integers $a_{n}$ grow superexponentially fast, in the sense that
\begin{equation}\label{eq:superexpan}
\lim_{n\to\infty}\frac{\log (a_{n+1}/a_{n})}{\log n}=\infty,
\end{equation}
which we assume in what follows. In view of the above discussion, we deduce that with probability one,
\[
\Hdim\left(\Gcal(\aseq,(n^{-\nu})_{n\geq 1})\cap G\right)=
\begin{cases}
\gamma & \text{if } \nu\leq 1 \\[2mm]
1/\nu+\gamma-1 & \text{if } 1<\nu\leq 1/(1-\gamma) \\[2mm]
-\infty & \text{if } \nu>1/(1-\gamma). \\[2mm]
\end{cases}
\]
An emblematic situation is when the compact set $G$ is the middle-third Cantor set $K$, which is regular with dimension $\kappa$ given by~(\ref{eq:dimCantor}). Furthermore, in the mere situation where $G$ is the whole circle $\T$, considering sequences $\aseq=(a_{n})_{n\geq 1}$ that grow superexponentially fast is also new. In those two cases, adopting notations that are more customary in the metric theory of Diophantine approximation, we may rewrite the previous result as follows.

\begin{thm}\label{thm:approxfrac}
Let $(a_{n})_{n\geq 1}$ be a sequence of positive integers such that~(\ref{eq:superexpan}) holds. Then, for Lebesgue almost every real $\alpha$ and for every real $\nu\geq 1$,
\[
\Hdim\left\{\xi\in \R \:\biggl|\: \|a_{n}\alpha-\xi\|<\frac{1}{n^{\nu}} \text{~for i.m.~} n\geq 1\right\}=\frac{1}{\nu} 
\]
and
\[
\Hdim\left\{\xi\in K\:\biggl|\: \|a_{n}\alpha-\xi\|<\frac{1}{n^{\nu}} \text{~for i.m.~} n\geq 1\right\}=\frac{1}{\nu}+\kappa-1
\]
if this value is nonnegative; otherwise, the latter set is empty.
\end{thm}

Note that a simple example of a sequence $(a_n)_{n \ge 1}$
for which~(\ref{eq:condgrowthan}) is verified is 
given by $a_{n}=\lfloor n^{(1+\rho+\eps)n}\rfloor$ for $n\geq 1$, where $\eps$ is
any fixed positive real number. In addition, one easily checks that the more stringent condition~(\ref{eq:superexpan}) is satisfied for instance by the sequences of the form $a_{n}=n^{n b_{n}}$, where $(b_{n})_{n\geq 1}$ is an auxiliary sequence of positive integers that monotonically diverges to infinity. 

Our approach fails when the condition~(\ref{eq:superexpan}) is not verified, because there may be too much dependence between the fractional parts $\{a_{n}\alpha\}$, $n \ge 1$,  
for typical values of $\alpha$. This is the case in particular when $(a_{n})_{n \geq 1}$ has a linear or geometric growth. In those cases, however, the situation is well understood if one is not interested in taking the intersection with the Cantor set. In fact, when $a_{n}=n$, it is shown in~\cite{Bugeaud:2003ye,Schmeling:2003} that, for every irrational real number $\alpha$ and every real number $\nu\geq 1$,
\[
\Hdim\left\{\xi\in\R\:\biggl|\: \|n\alpha-\xi\|<\frac{1}{n^{\nu}} \text{~for i.m.~} n\geq 1\right\}=\frac{1}{\nu}.
\]
The case in which $a_{n}=2^{n}$ has been investigated by Fan, Schmeling, and Troubetzkoy~\cite{Fan:2007fj}, and also by Liao and Seuret~\cite{Liao:2013fk}. In particular, these authors determined the value of
\[
\Hdim\left\{\xi\in\R\:\biggl|\: \|2^{n}\alpha-\xi\|<\frac{1}{n^{\nu}} \text{~for i.m.~} n\geq 1\right\}
\]
when the real number $\alpha$ is chosen according to a variety of invariant Gibbs measures associated with the doubling map.

%%%%%%%%%%%%%%%%%%%%%%%%%%
%%%%%%%%%%%%%%%%%%%%%%%%%%
\section{Concluding remarks and further problems}\label{sec:remprob}
%%%%%%%%%%%%%%%%%%%%%%%%%%
%%%%%%%%%%%%%%%%%%%%%%%%%%

%%%%%%%%%%%%%%%%%%%%%%%%%%
\subsection{Approximation by algebraic numbers of bounded degree}
%%%%%%%%%%%%%%%%%%%%%%%%%%

One natural way to extend the theorem of Jarn\'\i k and Besicovitch is the study of the approximation to real numbers by algebraic numbers of bounded degree. For $n\geq 1$, the accuracy with which real numbers are approximated by algebraic numbers of degree at most $n$ is measured by means of the exponents $w_{n}^{\ast}$, introduced in~1939 by Koksma \cite{Koksma:1939fk}. (Although we do not introduce Mahler's exponents $w_{n}$, we prefer to keep the standard notation $w_{n}^{\ast}$ for Koksma's exponents.)

Recall that the height $H(P)$ of an integer polynomial $P(X)$ is the maximum of the moduli of its coefficients, and the height $H(a)$ of an algebraic number $a$ is the height of its minimal polynomial over $\Z$. For any integer $n\geq 1$ and any real number $\xi$, the exponent $w_{n}^{\ast}(\xi)$ is defined as the supremum of the real numbers $w^{\ast}$ for which the inequality
\begin{equation}\label{eq:defwnstar}
0<|\xi-a|\leq H(a)^{-w^{\ast}-1}
\end{equation}
is satisfied for infinitely many algebraic numbers $a$ of degree at most $n$. Clearly, every real number $\xi$ satisfies
\[
\mu(\xi)=w_{1}^{\ast}(\xi)+1.
\]
This shows that the exponents $w_{n}^{\ast}$ with $n\geq 2$ extend in a natural way the irrationality exponent $\mu$.

The introduction of the exponent $-1$ in~(\ref{eq:defwnstar}) is explained on~\cite[p.~48]{Bugeaud:2004fk}. The reader is directed to this monograph for known results on the exponents $w_{n}^{\ast}$. We only mention here that $w_{n}^{\ast}(\xi)=\min\{n,d-1\}$ for every real algebraic number $\xi$ of degree $d$ and that Lebesgue almost all real numbers $\xi$ satisfy $w_{n}^{\ast}(\xi)=n$ for all $n \geq 1$. In~1970, Baker and Schmidt~\cite{Baker:1970jf} extended the theorem of Jarn\'\i k and Besicovitch to the exponents $w_{n}^{\ast}$. They established that, for every integer $n\geq 1$ and every real number $w^{\ast}\geq n$, the set
\[
\Ucal_{n}(w^{\ast})=\{\xi\in\R  \:|\: w_{n}^{\ast}(\xi)\geq w^{\ast}\}
\]
satisfies
\begin{equation}\label{eq:dimmnwn}
\Hdim\Ucal_{n}(w^{\ast})=\frac{n+1}{w^{\ast}+1}.
\end{equation}
Note that~(\ref{eq:JarnikBesicovitch}) and~(\ref{eq:dimmnwn}) coincide, as expected, for $n=1$. Some further metric properties of the sets $\Ucal_{n}(w^{\ast})$ were obtained in~\cite{Bugeaud:2004wc,Durand:2007uq}; in particular, it is proven in those two papers that the above sets are intersective in the sense of Falconer~\cite{Falconer:1994hx}.

The result due to Weiss that is mentioned at the very beginning of Section~\ref{sec:intro} was extended to the exponents $w_{n}^{\ast}$ by Kleinbock, Lindenstrauss and Weiss; they proved that, with respect to the standard measure on the middle-third Cantor set, almost all points $\xi$ satisfy
\[
\forall n\geq 1 \qquad w_{n}^{\ast}(\xi)=n,
\]
see~\cite[Proposition 7.10]{Kleinbock:2004fk}. This motivates the following open question.

\begin{pb}
Let $n\geq 1$ be an integer and $w^{\ast}\geq n$ be a real number. To determine the Hausdorff dimension of the set
\[
\Ucal_{n}(w^{\ast})\cap K=\{\xi\in K \:|\: w_{n}^{\ast}(\xi)\geq w^{\ast}\}.
\]
\end{pb}

As regards this problem, we believe that the following natural extension of Conjecture~\ref{conj:main} holds.

\begin{conj}\label{conj:mainwn}
For any integer $n\geq 1$ and any real number $w^{\ast}\geq n$, the set of points in the middle-third Cantor set which are approximable at order at least $w^{\ast}+1$ by algebraic numbers of degree at most $n$ satisfies
\begin{equation}\label{eq:conjmainwn}
\Hdim(\Ucal_{n}(w^{\ast})\cap K)
=\max\left\{\frac{n+1}{w^{\ast}+1}+\kappa-1,\frac{\kappa}{w^{\ast}+1}\right\}.
\end{equation}
\end{conj}

A noteworthy result towards this conjecture was established by Kristensen~\cite{Kristensen:2006ys}. Specifically, extending the covering argument used in~\cite{Pollington:2005fk,Weiss:2001mz}, he proved the following upper bound:
\begin{equation}\label{eq:upbnddimUmwK}
\Hdim(\Ucal_{n}(w^{\ast})\cap K) \leq \frac{2n \kappa}{w^{\ast}+1}.
\end{equation}
Furthermore, when $n$ is fixed and $h$ varies, the number of algebraic numbers with degree at most $n$ and height equal to $h$ that belong to the circle $\T$ is of the order of $h^{n}$. Therefore, in the light of the approach developed in Section~\ref{subsec:doubmet}, the following extension of Theorem~\ref{thm:genupbndMcalmualphaK} plainly holds: for Lebesgue almost every angle $\alpha\in\T$, we have both
\[
w_{n}^{\ast}(\xi)\leq\frac{n+1}{1-\kappa}-1
\]
for all points $\xi\in\alpha+K$, and
\[
\Hdim(\Ucal_{n}(w^{\ast})\cap(\alpha+K))\leq\frac{n+1}{w^{\ast}+1}+\kappa-1
\]
for all real numbers $w^{\ast}$ between $n$ and $(n+1)/(1-\kappa)-1$. Note that this last generic bound is much more stringent than~(\ref{eq:upbnddimUmwK}).

Let us mention that virtually all the ideas developed in this paper may be adapted to the setting of the approximation by algebraic numbers. In particular, as in Section~\ref{subsec:probconjmain}, one may define an appropriate probabilistic counterpart of the exponents $w_{n}^{\ast}$ and establish the corresponding version of Conjecture~\ref{conj:mainwn}. Likewise, the same heuristic arguments as those put forward in Section~\ref{subsec:heuristic} suggest that~(\ref{eq:conjcardPKmuj}) can be extended to
\[
\limsup_{j\to\infty}\frac{1}{j}\log_{3}\#\Acal^{n,w^{\ast},j}_{K}\leq\max\{n+1-(1-\kappa)(w^{\ast}+ 1),\kappa\},
\]
where $\Acal^{n,w^{\ast},j}_{K}$ denotes the set of all algebraic numbers $a\in\T$ with degree at most $n$ that satisfy both $3^{j}\leq H(a)<3^{j+1}$ and $\dist(a, K)<3^{-(w^{\ast}+1)j}$. Such an upper bound would obviously be in favor of the validity of Conjecture~\ref{conj:mainwn}.

%%%%%%%%%%%%%%%%%%%%%%%%%%
\subsection{A more general framework}
%%%%%%%%%%%%%%%%%%%%%%%%%%

All the number theoretical problems discussed above can be put in a same general framework. We consider the following question. Let $\xseq=(x_{n})_{n\geq 1}$ be a sequence of points in $\T$. Given a real number $\nu\geq 1$, let us consider the set
\[
\Hcal(\xseq,\nu)=\{\xi\in\T \:|\: \dist(\xi,x_{n})<n^{-\nu} \text{~for i.m.~} n\geq 1 \}.
\]
When the sequence $\xseq$ forms a regular system in the sense of~\cite[Chapter~5]{Bugeaud:2004fk}, we have
\[
\Hdim\Hcal(\xseq,\nu)=\frac{1}{\nu}.
\]
Note that this general statement includes~(\ref{eq:dimmnwn}) after having suitably numbered the algebraic numbers in $\T$ of degree at most $n$, see~\cite[Lemma~5.4]{Bugeaud:2004fk}.

The general problem that we are concerned with is the estimation of the Hausdorff dimension of the set
\[
\Hcal(\xseq,\nu)\cap G=\{\xi\in G \:|\: \dist(\xi,x_{n})<n^{-\nu} \text{~for i.m.~} n\geq 1 \},
\]
where $G$ is a compact subset of the circle $\T$ which, for simplicity, is supposed to be regular with dimension $\gamma\in(0,1)$. Here, we take the intersection of two null sets of very different nature. The set $G$ is compact and nowhere dense, whereas when $\xseq$ forms a regular system, the set $\Hcal(\xseq,\nu)$ is an intersective set in the sense of Falconer~\cite{Falconer:1994hx}, see~\cite{Bugeaud:2004wc,Durand:2007uq}. Even giving an accurate upper bound on the Hausdorff dimension of the intersection set $\Hcal(\xseq,\nu)\cap G$ is challenging.

\begin{pb}\label{pb:prob2}
Find reasonable conditions under which one can prove either of the upper bounds
\[
\Hdim(\Hcal(\xseq,\nu)\cap G)\leq\frac{\gamma}{\nu}
\qquad\text{or}\qquad
\Hdim(\Hcal(\xseq,\nu)\cap G)\leq\frac{1}{\nu}+\gamma-1.
\]
\end{pb}

Note in passing that the latter bound is more stringent than the former. Moreover, a preliminary step towards the first bound in Problem~\ref{pb:prob2} would be to understand for which sequences $\xseq$ one can apply the arguments of~\cite{Pollington:2005fk,Weiss:2001mz}. As regards Problem~\ref{pb:prob2}, the only general result that may be deduced from the present paper is again an extension of Theorem~\ref{thm:genupbndMcalmualphaK}, namely: for Lebesgue almost every angle $\alpha\in\T$,
\[
\Hdim(\Hcal(\xseq,\nu)\cap (\alpha+G))\leq\frac{1}{\nu}+\gamma-1\,;
\]
in particular, the intersection set is empty if the last bound is negative. This means that the second bound in Problem~\ref{pb:prob2} holds when the compact set $G$ is rotated in a generic manner. Moreover, inspecting the proof of Theorem~\ref{thm:genupbndMcalmualphaK}, we see that the above result still holds when $G$ is not regular but only satisfies $\prepack^{g}(G)<\infty$.

Restricting to the case where $G$ is the middle-third Cantor set $K$, one may also point out the following question.

\begin{pb}\label{pb:prob3}
Compare the values of the Hausdorff dimensions
\[
\Hdim\Hcal(\xseq,\nu) \qquad\text{and}\qquad \Hdim(\Hcal(\xseq,\nu)\cap K).
\]
\end{pb}

The following two examples show that there is no hope of getting a general answer to Problem~\ref{pb:prob3}; we in fact have two extremal cases:
\begin{enumerate}
\item Assume that $\xseq=(x_{n})_{n\geq 1}$ denotes the natural enumeration of the rational numbers in $\T$ of the form $p/3^{j}$ such that $\gcd(3,p)=1$ and $p$ has only digits $0$ and $2$ in its ternary representation. Then, the denominator of $x_{n}$ is of the order of $n^{1/\kappa}$, and Corollary~1 in~\cite{Levesley:2007pd} implies that for all $\nu\geq 1/\kappa$,
\[
\Hdim\Hcal(\xseq,\nu)=\Hdim(\Hcal(\xseq,\nu)\cap K)=\frac{1}{\nu}. 
\]
Note that we even have $\Hcal(\xseq,\nu)\cap K = \Hcal(\xseq,\nu)$.
\item Now, assume that $\xseq=(x_{n})_{n\geq 1}$ is the natural enumeration of the rational numbers in $\T$ of the form $p/3^{j}-1/(2\cdot 3^{j})$ such that $\gcd(3,p)=1$ and $p$ has only digits $0$ and $2$ in its ternary representation. Then, the denominator of $x_{n}$ is still of the order of $n^{1/\kappa}$ and for all $\nu\geq 1/\kappa$,
\[
\Hdim\Hcal(\xseq,\nu)=\frac{1}{\nu}. 
\]
However, each point $x_{n}$ is very far from $K$, namely, at a distance of the order of $n^{1/\kappa}$. Thus, in this case, the set $\Hcal(\xseq,\nu)\cap K$ is empty for $\nu>1/\kappa$. 
\end{enumerate}

%%%%%%%%%%%%%%%%%%%%%%%%%%
%%%%%%%%%%%%%%%%%%%%%%%%%%
\section{Proof of the main results}\label{sec:proofmain}
%%%%%%%%%%%%%%%%%%%%%%%%%%
%%%%%%%%%%%%%%%%%%%%%%%%%%

%%%%%%%%%%%%%%%%%%%%%%%%%%
\subsection{Proof of Theorem~\ref{thm:majsizegene}}\label{sec:proofthmmajsizegene}
%%%%%%%%%%%%%%%%%%%%%%%%%%

Recall that the set $I_{n,k}$ is defined in Section~\ref{subsec:randomgenelow} as the image of the interval $[k/q_{n},(k+1)/q_{n})$ under the projection modulo one. Throughout the proof of Theorem~\ref{thm:majsizegene}, we choose $q_{n}$ to be equal to the integer part of $1/r_{n}$, that is, $q_{n}=\lfloor 1/r_{n}\rfloor$. The circle $\T$ may thus be seen as the disjoint union over $k\in\{0,\ldots,q_{n}-1\}$ of the sets $I_{n,k}$. Then, let $K_{n}(G)$ denote the set of integers $k$ for which the compact set $G$ intersects $I_{n,k}$. In addition, let $I_{n,k}'$ be the closed subinterval of $\T$ which has the same midpoint as $I_{n,k}$ and is three times larger. It is now easy to check that
\begin{equation}\label{eq:GcapBxrsub}
\forall n\geq 1 \qquad \opball{X_{n}}{r_{n}}\cap G\subseteq\bigcup_{k\in K_{n}(G) \atop X_{n}\in I_{n,k}'} I_{n,k},
\end{equation}
where $\opball{X_{n}}{r_{n}}$ denotes the open interval centered at $X_{n}$ with radius $r_{n}$.

Furthermore, note that for any $k\in K_{n}(G)$, there exists a point $\xi_{k}\in G\cap I_{n,k}$. Thus, there exists a subset $K_{n}'(G)$ of $K_{n}(G)$ with cardinality at least $\#K_{n}(G)/2$ such that $\dist(\xi_{k},\xi_{k'})>1/q_{n}$ for any distinct $k$ and $k'$ in $K_{n}'(G)$. In view of~\cite[Lemma~4]{Tricot:1982fk}, the finiteness of $\prepack^{g}(G)$ ensures that there exists a finite Borel measure $\chi$ such that $\chi(\opball{x}{r})\geq g(r)$ for all $x\in G$ and $r\in(0,1)$. As a result,
\[
\chi(\T)\geq\sum_{k\in K_{n}'(G)} \chi(\opball{\xi_{k}}{1/(2q_{n})})\geq\frac{\#K_{n}(G)}{2}g\left(\frac{1}{2q_{n}}\right).
\]
Since the gauge function $g$ is nondecreasing and doubling, we deduce that
\begin{equation}\label{eq:majcardKnG}
\exists C>0 \quad \forall n\geq 1 \qquad \#K_{n}(G)\leq \frac{C}{g(r_{n})}.
\end{equation}

We may now prove the second statement of the theorem. Let us consider a doubling gauge function $h$ such that the series $\sum_{n} h(r_{n})r_{n}/g(r_{n})$ converges. For any real $\delta\in (0,1/2)$ and any integer $n_{0}\geq 1$, the inclusion~(\ref{eq:GcapBxrsub}) ensures that the set $\Ecal(\Xseq,\rseq)\cap G$ is covered by the intervals $I_{n,k}$ indexed by the integers $n\geq n_{0}$ and $k\in K_{n}(G)$ for which $X_{n}\in I_{n,k}'$. All these intervals have diameter $1/q_{n}$, which is smaller than $\delta$ for $n$ large enough, due to the convergence of the aforementioned series. As a consequence,
\[
\hau^{h}_{\delta}(\Ecal(\Xseq,\rseq)\cap G)\leq\sum_{n=n_{0}}^{\infty}h\left(\frac{1}{q_{n}}\right)\sum_{k\in K_{n}(G)} \ind_{\{X_{n}\in I_{n,k}'\}}.
\]
Making use of~(\ref{eq:majcardKnG}), the fact that the variables $X_{n}$ are uniformly distributed, and the fact that $h$ is doubling, we deduce that
\[
\esp[\hau^{h}_{\delta}(\Ecal(\Xseq,\rseq)\cap G)]
\leq\sum_{n=n_{0}}^{\infty}h\left(\frac{1}{q_{n}}\right)\#K_{n}(G)\frac{3}{q_{n}}
\leq 6CC'\sum_{n=n_{0}}^{\infty}\frac{h(r_{n})r_{n}}{g(r_{n})},
\]
where $C'$ depends on $h$ only. We may now let $n_{0}$ tend to infinity, and then let $\delta$ tend to zero. Fatou's lemma then implies that $\hau^{h}(\Ecal(\Xseq,\rseq)\cap G)$ has mean zero, and the second part of the theorem follows.

The constant function equal to one is not, strictly speaking, a gauge function in the sense of our definition. However, this function may be used instead of $h$ above, thereby leading to the first statement of the theorem. Indeed, in that situation, the Hausdorff measure is just the counting measure, and the previous arguments imply that $\#(\Ecal(\Xseq,\rseq)\cap G)$ has mean zero.

%%%%%%%%%%%%%%%%%%%%%%%%%%
\subsection{Proof of Theorem~\ref{thm:nonemptygene}}\label{sec:proofthmnonemptygene}
%%%%%%%%%%%%%%%%%%%%%%%%%%

By virtue of Frostman's lemma, see for instance~\cite[Theorem~8.8]{Mattila:1995fk}, the positivity of $\hau^{g}(G)$ implies that there exists a Borel measure $\chi$ such that
\begin{equation}\label{eq:Frostman}
\forall x\in\T \quad \forall r\in(0,1) \qquad \chi(\opball{x}{r})\leq g(r).
\end{equation}
Moreover, the support of the measure $\chi$ is a nonempty compact subset of $G$ that is denoted by $G'$ in what follows.

Thanks to a Baire category argument appearing in~\cite[p.~12]{Khoshnevisan:2000fj}, we only need to show that for any fixed open interval $I$ of the circle that intersects $G'$, the event
\[
E_{I}=\left\{\dist(X_{n},G'\cap I)<r_{n} \text{~for i.m.~} n\geq 1\right\}
\]
holds with probability one. Here, $\dist(X_{n},G'\cap I)$ denotes the distance from the point $X_{n}$ to the set $G'\cap I$. Indeed, assuming that this holds and letting $I$ run through a countable base of open intervals that generate the topology on $\T$, we deduce that, with probability one, all the events $E_{V}$, for $V$ running through the open sets that intersect $G'$, hold simultaneously. As a result, with probability one, for any open subset $V$ of the circle and any integer $u\geq 1$,
\[
G'\cap V\neq\emptyset \qquad\Longrightarrow\qquad G'\cap V\cap\bigcup_{n=u}^{\infty}\opball{X_{n}}{r_{n}}\neq\emptyset,
\]
which means that the above union is dense in the complete metric space $G'$. The Baire category theorem then ensures that the set $\Ecal(\Xseq,\rseq)\cap G'$ is almost surely dense in $G'$, and therefore nonempty. It follows that the set $\Ecal(\Xseq,\rseq)\cap G$ is almost surely nonempty as well.

Let us now consider an open interval $I$ that intersects $G'$ and show that the event $E_{I}$ holds with probability one. We shall make use of the same notations as in the proof of Theorem~\ref{thm:majsizegene}, except that we choose $q_{n}$ to be equal to $\lceil 1/r_{n}\rceil$, where $\lceil\,\cdot\,\rceil$ stands for the ceiling function. In addition to these notations, let $I_{n}(G)$ denote the union over $k\in K_{n}(G)$ of the intervals $I_{n,k}$. On easily checks that for every $n\geq 1$,
\[
X_{n}\in I_{n}(G) \qquad\Longrightarrow\qquad \dist(X_{n},G'\cap I)<r_{n}.
\]
Therefore, it suffices to show that with probability one, $X_{n}\in I_{n}(G)$ for infinitely many integers $n\geq 1$.

To this purpose, let us begin by observing that the complement of the set $I_{n}(G)$ is the union over $k\in K_{n}(G)^{\complement}$ of the sets $I_{n,k}$, where $K_{n}(G)^{\complement}$ denotes the complement of $K_{n}(G)$ in $\{0,\ldots,q_{n}-1\}$. Thus, for $u\leq v$,
\[
\prob\left(\bigcap_{n=u}^{v} \{X_{n}\not\in I_{n}(G)\}\right)
=\sum_{k_{u},\ldots,k_{v}} \prob\left(\bigcap_{n=u}^{v} \{X_{n}\in I_{n,k_{n}}\}\right)
\]
where each index $k_{n}$ arising in the sum runs over the set $K_{n}(G)^{\complement}$. All the terms in this sum are bounded from above by $\Theta(\Xseq,\Bcal(\rseq))/(q_{u}\cdot\ldots\cdot q_{v})$, so that the whole sum is bounded from above by 
\[
\Theta(\Xseq,\Bcal(\rseq)) \cdot \prod_{n=u}^{v} \frac{\#K_{n}(G)^{\complement}}{q_{n}}. 
\]
Using that $1-y\leq\ee^{-y}$ and $q_{n}<2/r_{n}$, we get
\[
\prod_{n=u}^{v} \frac{\#K_{n}(G)^{\complement}}{q_{n}}\leq\exp\left(-\frac{1}{2}\sum_{n=u}^{v} \#K_{n}(G)\, r_{n}\right).
\]
In addition, recall that the intersection of the sets $G$ and $I_{n,k}$ is nonempty if and only if $k$ belongs to $K_{n}(G)$. Moreover, in that situation, this intersection contains a point $x_{n,k}$ and is therefore included in the closed interval $\clball{x_{n,k}}{r_{n}}$ with radius $r_{n}$ centered at this point. With the help of~(\ref{eq:Frostman}), we deduce that
\[
\chi(\T)=\sum_{k\in K_{n}(G)}\chi(G\cap I_{n,k})\leq\sum_{k\in K_{n}(G)}\chi(\clball{x_{n,k}}{r_{n}})\leq\#K_{n}(G)\,g(r_{n}).
\]
As a consequence, the cardinality of the set $K_{n}(G)$ is bounded from below by $\chi(\T)/g(r_{n})$. We infer that
\[
\prob\left(\bigcup_{n_{0}=1}^{\infty}\bigcap_{n=n_{0}}^{\infty} \{X_{n}\not\in I_{n}(G)\}\right)
\leq\sum_{n_{0}=1}^{\infty}\exp\left(-\frac{\chi(\T)}{2}\sum_{n=n_{0}}^{\infty} \frac{r_{n}}{g(r_{n})}\right)=0,
\]
in view of the divergence of the series $\sum_{n} r_{n}/g(r_{n})$. The result follows.

%%%%%%%%%%%%%%%%%%%%%%%%%%
\subsection{Proof of Lemma~\ref{lem:codim}}\label{subsec:prooflemcodim}
%%%%%%%%%%%%%%%%%%%%%%%%%%

The proof relies on the existence of a family of compact sets $Q_{g}$ indexed by the gauge functions such that Proposition~\ref{prp:propQg} below holds. These compact sets are obtained through a slight generalization of Mandelbrot's fractal percolation process that we introduce and study in Section~\ref{sec:fracperc}. In the next statement, we make use of the notation $h\prec_{\ph} g$ defined by~(\ref{eq:defprecph}), and we also write $h\pprec g$ to indicate that two gauge functions $g$ and $h$ satisfy
\[
\sum_{j=1}^{\infty} g(2^{-j})\left(\frac{1}{h(2^{-j})}-\frac{1}{h(2^{-(j-1)})}\right)<\infty.
\]

\begin{prp}\label{prp:propQg}
There exists a family of compact sets $Q_{g}\subseteq\T$ indexed by the gauge functions such that for any set $E\subseteq\T$, and any gauge function $g$, the following properties hold:
\begin{enumerate}
\item\label{item:prp:propQg:1} Let us assume that $\hau^{g}(E)=0$. Then,
\[
\as \qquad E\cap Q_{g}=\emptyset.
\]
\item\label{item:prp:propQg:2} Let us assume that $E$ is Borel and $\hau^{g}(E)>0$. Then, for every gauge functions $h$ and $\ph$,
\[
\left\{\begin{array}{ccl}
h\pprec g & \qquad\Longrightarrow\qquad & \prob(E\cap Q_{h}\neq\emptyset)>0\\[3mm]
h\prec_{\ph} g & \qquad\Longrightarrow\qquad & \prob(\hau^{\ph}(E\cap Q_{h})>0)>0.
\end{array}\right.
\]
\end{enumerate}
\end{prp}

The above result is established in Section~\ref{sec:fracperc}; this is actually a straightforward consequence of Lemmas~\ref{lem:hitperconull}--\ref{lem:hauperco} therein. Let us now give the proof of Lemma~\ref{lem:codim}. Let $\ph$ denote a gauge function, and let $E$ be a random subset of the circle which intersects almost surely every fixed compact set having positive Hausdorff $\ph$-measure. Now, let $G$ denote a compact subset of the circle, and let $g$ be a gauge function such that $\hau^{g}(G)>0$. In addition, let $h$ be a gauge function such that $h\prec_{\ph} g$. Our purpose is now to show that $\hau^{h}(E\cap G)>0$ almost surely.

To proceed, let $(Q_{h}^{n})_{n\geq 1}$ denote a sequence of independent copies of the fractal percolation set $Q_{h}$ that are also independent of the random set $E$, and let
\[
\widehat{Q}_{h}=\bigcup_{n=1}^{\infty} Q_{h}^{n}.
\]
Proposition~\ref{prp:propQg}(\ref{item:prp:propQg:2}) ensures that each set $G\cap Q_{h}^{n}$ has Hausdorff $\ph$-measure zero with a probability which does not depend on $n$ and is smaller than one. Thus, in view of the subadditivity of Hausdorff measures, we have
\[
\as \qquad \hau^{\ph}(G\cap\widehat{Q}_{h})>0.
\]
Therefore, in view of~\cite[Theorem~4.10]{Falconer:2003oj}, the set $G\cap\widehat{Q}_{h}$ almost surely contains a compact set with positive Hausdorff $\ph$-measure. This compact subset thus intersects the random set $E$ with probability one. Therefore,
\[
\as \qquad E\cap G\cap\widehat{Q}_{h}\neq\emptyset.
\]
On top of that, if $\hau^{h}(E\cap G)=0$, then Proposition~\ref{prp:propQg}(\ref{item:prp:propQg:1}) ensures that the probability that the set $E\cap G$ intersects any of the copies $Q_{h}^{n}$ is equal to zero. As a consequence,
\[
\as \qquad E\cap G\cap\widehat{Q}_{h}\neq\emptyset \quad\Longrightarrow\quad \hau^{h}(E\cap G)>0,
\]
and Lemma~\ref{lem:codim} follows.

%%%%%%%%%%%%%%%%%%%%%%%%%%
%%%%%%%%%%%%%%%%%%%%%%%%%%
\section{A generalized fractal percolation process}\label{sec:fracperc}
%%%%%%%%%%%%%%%%%%%%%%%%%%
%%%%%%%%%%%%%%%%%%%%%%%%%%

This section is devoted to the construction and the study of the family of compact sets $Q_{g}$ that we use in the proof of Lemma~\ref{lem:codim}, see Section~\ref{subsec:prooflemcodim} above. These sets are obtained by dint of a slight generalization of Mandelbrot's fractal percolation process. To begin with, recall that the infinite complete binary tree may naturally be encoded by the set
\[
\Tcal=\bigcup_{j=0}^{\infty}\{0,1\}^{j}.
\]
Here, we adopt the convention that $\{0,1\}^{0}$ is reduced to the singleton containing only the root $\varnothing$. Specifically, every node $u\in\Tcal$ with generation $\gen{u}=j$ may be seen as a finite word $u=u_{1}\ldots u_{j}$ over the alphabet $\{0,1\}$, with child nodes the two words $u_{1}\ldots u_{j}0$ and $u_{1}\ldots u_{j}1$, and with parent node the word $\pare{u}=u_{1}\ldots u_{j-1}$ in the case where $j$ is positive. The tree structure is then recovered by endowing the vertex set $\Tcal$ with the arcs $(\pare{u},u)$, for $u\neq\varnothing$.

Now, given a gauge function $g$, let us consider the following inhomogeneous percolation process on the edges of the above tree: the edge connecting a given node $u\neq\varnothing$ to his parent is retained with a probability equal to $g(2^{-\gen{u}})/g(2^{-\gen{u}+1})$, independently of the other edges. Then, we say that a node $u$ survives the percolation if all the edges between $u$ and the root $\varnothing$ are retained. We always assume that the root itself survives the percolation. The random subtree of $\Tcal$ composed by the nodes that survive the percolation is denoted by $\Tcal_{g}$. Note that every given node $u$ survives the percolation with probability $g(2^{-\gen{u}})/g(1)$.

Furthermore, recall that the vertices of the tree $\Tcal$ lead to a natural parametrization of the dyadic intervals of the circle $\T$. In fact, the dyadic interval associated with a node $u=u_{1}\ldots u_{j}$ in $\Tcal$ is the image, denoted by $\lambda(u)$, of the interval $(u_{1}2^{-1}+\ldots+u_{j}2^{-j})+[0,2^{-j})$ under the projection modulo one. In addition, the interval associated with the root is chosen to be the whole circle, that is, $\lambda(\varnothing)=\T$. This enables us to consider the random compact subset of the circle
\[
Q_{g}=\bigcap_{j=1}^{\infty}\downarrow\bigcup_{u\in\Tcal_{g}\atop\gen{u}=j}\adh{\lambda(u)},
\]
where $\adh{\,\cdot\,}$ stands for closure. The set $Q_{g}$ may be seen as an extension to the inhomogeneous setting of the compact set obtained through Mandelbrot's fractal percolation process, see~\cite{Durand:2009fk} and the references therein.

Let us recall that Proposition~\ref{prp:propQg} above contains all the important properties satisfied by the sets $Q_{g}$ that we use in the proof of Lemma~\ref{lem:codim}. This proposition may naturally be split into three separate lemmas that we now state and prove. The first two lemmas discuss the probability with which the random set $Q_{g}$ intersects a given subset of the circle.

\begin{lem}\label{lem:hitperconull}
For any set $E\subseteq\T$ and any gauge function $g$,
\[
\hau^{g}(E)=0 \qquad\Longrightarrow\qquad \as \quad E\cap Q_{g}=\emptyset.
\]
\end{lem}

\begin{proof}
Let $\delta>0$ and let $(U_{n})_{n\geq 1}$ be a sequence of subsets of $\T$ such that $E\subseteq\bigcup_{n} U_{n}$ and $\diam{U_{n}}<\delta$ for all $n\geq 1$. Let $\Ncal_{0}$ be the set of all integers $n\geq 1$ for which the diameter of $U_{n}$ vanishes, and let $\Ncal_{1}$ denote its complement in $\N$. Then, the set $E$ may be decomposed as the union of the sets
\[
E_{0}=E\cap\bigcup_{n\in\Ncal_{0}} U_{n} \qquad\text{and}\qquad E_{1}=E\cap\bigcup_{n\in\Ncal_{1}} U_{n}.
\]

On the one hand, it is easy to check that any point of the circle that is fixed in advance belongs to the set $Q_{g}$ with probability zero. Since the set $E_{0}$ is at most countable, it follows that its intersection with $Q_{g}$ is almost surely empty.

On the other hand, for any integer $n\in\Ncal_{1}$, the set $U_{n}$ is contained in four dyadic intervals of length at most $\diam{U_{n}}$. Accordingly, there exists a family of nodes $u^{n,i}$ in $\Tcal$, with $n\in\Ncal_{1}$ and $i\in\{1,2,3,4\}$, such that
\[
\left\{\begin{array}{l}
U_{n}\subseteq\lambda(u^{n,1})\cup\lambda(u^{n,2})\cup\lambda(u^{n,3})\cup\lambda(u^{n,4})\\[2mm]
\max\{\diam{\lambda(u^{n,1})},\diam{\lambda(u^{n,2})},\diam{\lambda(u^{n,3})},\diam{\lambda(u^{n,4})}\}\leq \diam{U_{n}}
\end{array}\right.
\]
for all $n\in\Ncal_{1}$. As a result, $E_{1}\cap Q_{g}$ is covered by the sets $\lambda(u^{n,i})\cap Q_{g}$ for $n\in\Ncal_{1}$ and $i\in\{1,2,3,4\}$. Note that, with probability one, the set $Q_{g}$ cannot contain any dyadic point, that is, any point of the form $k2^{-j}$. Thus, this last intersection is empty if $u^{n,i}$ does not survive the percolation. This implies that one of these nodes necessarily survives the percolation when the set $E_{1}\cap Q_{g}$ is nonempty. Hence,
\[
\prob(E_{1}\cap Q_{g}\neq\emptyset)
\leq\sum_{n\in\Ncal_{1} \atop i\in\{1,2,3,4\}} \prob(u^{n,i}\in\Tcal_{g})
=\sum_{n\in\Ncal_{1} \atop i\in\{1,2,3,4\}} \frac{g(2^{-\gen{u^{n,i}}})}{g(1)}
\leq \frac{4}{g(1)} \sum_{n=1}^{\infty} g(\diam{U_{n}}).
\]

Taking the infimum over all the possible coverings $(U_{n})_{n\geq 1}$, and then letting $\delta$ go to zero, we deduce that
\[
\prob(E\cap Q_{g}\neq\emptyset)\leq\prob(E_{0}\cap Q_{g}\neq\emptyset)+\prob(E_{1}\cap Q_{g}\neq\emptyset)\leq \frac{4}{g(1)}\hau^{g}(E),
\]
and the result follows.
\end{proof}

\begin{lem}\label{lem:hitpercopos}
For any Borel set $E\subseteq\T$ and any gauge function $g$,
\[
\hau^{g}(E)>0 \qquad\Longrightarrow\qquad \forall h\pprec g \quad \prob(E\cap Q_{h}\neq\emptyset)>0.
\]
\end{lem}

\begin{proof}
Since $\hau^{g}(E)$ is positive, Frostman's lemma implies that there exists a Borel measure $\chi$ with support included in $E$ such that~(\ref{eq:Frostman}) holds. For any node $u\in\Tcal$, set $\psi(u)=\chi(\lambda(u))$. Moreover, for any integer $j\geq 0$, set $h_{j}=h(2^{-j})$ and
\[
Z_{j}=\frac{1}{h_{j}}\sum_{u\in\Tcal_{h} \atop \gen{u}=j}\psi(u),
\]
and let $\Gcal_{j}$ denote the $\sigma$-algebra generated by the events $\{u\in\Tcal_{h}\}$ for $\gen{u}\leq j$. It is then easy to check that $(Z_{j})_{j\geq 0}$ is a nonnegative martingale with respect to the filtration $(\Gcal_{j})_{j\geq 0}$, thereby converging almost surely to some random variable $Z_{\infty}\in L^{1}$. Furthermore, for any integer $j\geq 1$,
\[
\esp[Z_{j}^{2}]=\frac{1}{h_{j}^{2}}\sum_{u,v\in\Tcal \atop \gen{u}=\gen{v}=j}\psi(u)\psi(v)\,\prob(u\in\Tcal_{h}\text{ and }v\in\Tcal_{h}).
\]
The probability that two nodes $u$ and $v$ both survive the percolation is clearly equal to $(h_{\gen{u}}h_{\gen{v}})/(h_{0}h_{\gen{u\wedge v}})$, where $u\wedge v$ denotes their lowest common ancestor in the tree $\Tcal$. As a consequence,
\[
\esp[Z_{j}^{2}]=\sum_{w\in\Tcal \atop \gen{w}\leq j}\frac{1}{h_{0}h_{\gen{w}}}\sum_{\gen{u}=\gen{v}=j \atop u\wedge v=w}\psi(u)\psi(v)
\]
Note that the inner sum is equal to $\psi(w)^{2}$ if the node $w$ has generation $j$, and to $2\psi(w0)\psi(w1)$ if $w$ has generation less than $j$. Therefore,
\[
\esp[Z_{j}^{2}]=\sum_{w\in\Tcal \atop \gen{w}\leq j}\frac{1}{h_{0}h_{\gen{w}}}\left(\psi(w)^{2}-\ind_{\{\gen{w}\leq j-1\}}(\psi(w0)^{2}+\psi(w1)^{2})\right),
\]
from which it follows that
\[
\esp[Z_{j}^{2}]=\frac{\psi(\varnothing)^{2}}{h_{0}^{2}}+\frac{1}{h_{0}}\sum_{w\in\Tcal \atop 0<\gen{w}\leq j}\psi(w)^{2}\left(\frac{1}{h_{\gen{w}}}-\frac{1}{h_{\gen{w}-1}}\right).
\]
In view of~(\ref{eq:Frostman}) and the fact that $h\pprec g$, we deduce that
\[
\sup_{j\geq 0}\esp[Z_{j}^{2}]\leq\frac{\chi(\T)^{2}}{h(1)^{2}}+\frac{\chi(\T)}{h(1)}\sum_{j=1}^{\infty}g(2^{-j})\left(\frac{1}{h(2^{-j})}-\frac{1}{h(2^{-(j-1)})}\right)<\infty.
\]
This ensures that the martingale $(Z_{j})_{j\geq 0}$ converges to $Z_{\infty}$ in $L^{2}$. In particular, the expectation of $Z_{\infty}$ is equal to that of $Z_{0}$, specifically, $\chi(\T)/h(1)$, so that $Z_{\infty}$ is positive with positive probability. On top of that, note that if $Z_{\infty}$ is positive, then for any integer $j\geq 0$, there is a node $u\in\Tcal_{h}$ with generation equal to $j$ such that $\lambda(u)$ intersects the support of the measure $\chi$. In that case, Cantor's intersection theorem ensures that the limit set $Q_{h}$ intersects $E$, and the result follows.
\end{proof}

The third and last lemma about the sets $Q_{g}$ concerns the size of their intersection with a given Borel subset of the circle.

\begin{lem}\label{lem:hauperco}
For any Borel set $E\subseteq\T$ and any gauge functions $g$ and $\ph$,
\[
\hau^{g}(E)>0 \qquad\Longrightarrow\qquad \forall h\prec_{\ph} g \quad \prob(\hau^{\ph}(E\cap Q_{h})>0)>0.
\]
\end{lem}

\begin{proof}
Let us assume that the set $E$ has positive Hausdorff $g$-measure, and that $h\prec_{\ph} g$. Then, $g/h$ coincides with a gauge function $\psi$ satisfying $\ph\pprec\psi$. In particular, $h\ph\pprec g$ and Lemma~\ref{lem:hitpercopos} implies that the random set $Q_{h\ph}$ intersects $E$ with positive probability. Furthermore, it is easy to see that the set $Q_{h\ph}$ is distributed as $Q_{h}\cap Q_{\ph}$, where $Q_{h}$ and $Q_{\ph}$ are independent. Thus, the set $Q_{\ph}$ intersects $E\cap Q_{h}$ with positive probability as well. The result now follows from Lemma~\ref{lem:hitperconull}.
\end{proof}

\end{document}